\documentclass{amsart}
\usepackage[utf8]{inputenc}
\usepackage{mathrsfs} 

\usepackage{amssymb}
\usepackage{bm}
\usepackage{graphicx}
\usepackage[centertags]{amsmath}
\usepackage{amsfonts}
\usepackage{amsthm}

\usepackage{enumerate}
\usepackage{tocvsec2}
\usepackage{xcolor}
\usepackage[margin=1in]{geometry}

\newtheorem{theorem}{Theorem}[section]
\newtheorem*{theorem*}{Theorem}

\newtheorem{corollary}[theorem]{Corollary}
\newtheorem{lemma}[theorem]{Lemma}
\newtheorem{rem}[theorem]{Remark}

\newtheorem{proposition}[theorem]{Proposition}

\theoremstyle{definition}

\newcommand{\ee}{\varepsilon}
\newcommand{\nn}{\mathbb{N}}
\newcommand{\rr}{\mathbb{R}}

\begin{document}

\title{The complexity of some ordinal determined classes of operators}
\author{R.M. Causey}

\begin{abstract} We compute the complexity of the classes of  operators $\mathfrak{G}_{\xi, \zeta}\cap \mathcal{L}$ and $\mathfrak{M}_{\xi, \zeta}\cap \mathcal{L}$ in the coding of operators between separable Banach spaces. We also prove the non-existence of universal factoring operators for both $\complement \mathfrak{G}_{\xi, \zeta}$ and $\complement \mathfrak{M}_{\xi, \zeta}$. The latter result is an ordinal extension of a result of Johnson and Girardi.  \end{abstract}



\maketitle

\section{Introduction}

In this work, we investigate the complexity of some recently isolated operator ideals from two different points of view. The first point of view is by the classical search for a universal factoring operator for the complement of the ideal. The second point of view makes use of descriptive set theory and the coding $\mathcal{L}$ of the class of operators between separable Banach spaces first given in \cite{BF}.    The ideals of interest are ordinal-defined classes which are related to three important ideals: The weak Banach-Saks operators $\mathfrak{wBS}$, the completely continuous operators $\mathfrak{V}$, and the class $\mathfrak{DP}$ whose space ideal is the class of spaces with the Dunford-Pettis property. Each of these classes is defined by the behavior of weakly null sequences. Therefore it is natural to use the weakly null hierarchy defined by the Banach-Saks index of a weakly null sequence defined in \cite{AMT} to define quantified classes.    We give the the formal definition of $\xi$-weakly null in Section $2$.  Heuristically, given a weakly null sequence $(x_n)_{n=1}^\infty$, an ordinal assignment $\mathcal{BS}((x_n)_{n=1}^\infty)$ is defined which measures how weakly null the sequence $(x_n)_{n=1}^\infty$ is.    Sequences with smaller Banach-Saks index are ``more'' weakly null than sequences with larger index.   Given a Banach space $X$, we can define for each $0\leqslant \xi\leqslant \omega_1$ the set $\mathcal{WN}_\xi(X)$ to be the set of all weakly null sequences $(x_n)_{n=1}^\infty$ in $X$ with $\mathcal{BS}((x_n)_{n=1}^\infty)\leqslant \xi$.    The  properties of these classes relevant to this work are summarized in the following items. \begin{enumerate}[(i)]\item $\mathcal{WN}_0(X)$ consists of the norm null sequences in $X$, \item $\mathcal{WN}_1(X)$ consists of those sequences in $X$ such that every subsequence has a further subsequence whose Cesaro means converge to zero in norm. \item $\mathcal{WN}_{\omega_1}(X)=\cup_{\xi<\omega_1}\mathcal{WN}_\xi(X)$ is the set of all weakly null sequences in $X$. \end{enumerate}

Let us recall definitions of the classes $\mathfrak{wBS}, \mathfrak{V}$, and $\mathfrak{DP}$ using the notation from the previous paragraph The class $\mathfrak{wBS}$ is the class of all operators $A:X\to Y$ such that for every $(x_n)_{n=1}^\infty\in \mathcal{WN}_{\omega_1}(X)$, $(Ax_n)_{n=1}^\infty\in \mathcal{WN}_1(Y)$.    The class $\mathfrak{V}$ is the class of all operators $A:X\to Y$ such that for every $(x_n)_{n=1}^\infty \in\mathcal{WN}_{\omega_1}(X)$, $(Ax_n)_{n=1}^\infty \in \mathcal{WN}_0(Y)$.     The class $\mathfrak{DP}$ is the class of all operators $A:X\to Y$ such that for each  $(x_n)_{n=1}^\infty\in \mathcal{WN}_{\omega_1}(X)$ and $(y^*_n)_{n=1}^\infty\in \mathcal{WN}_{\omega_1}(Y^*)$, $\lim_n y^*_n(Ax_n)=0$.      Now for $0\leqslant \zeta,\xi\leqslant \omega_1$, we let $\mathfrak{G}_{\xi, \zeta}$ denote the class of all operators $A:X\to Y$ such that for each $(x_n)_{n=1}^\infty\in \mathcal{WN}_\xi(X)$, $(Ax_n)_{n=1}^\infty \in \mathcal{WN}_\zeta(Y)$.  Then $\mathfrak{wBS}=\mathfrak{G}_{\omega_1, 1}$ and $\mathfrak{V}=\mathfrak{G}_{\omega_1, 0}$.   It is easily verified that if $(x_n)_{n=1}^\infty\in \mathcal{WN}_\xi(X)$, then $(Ax_n)_{n=1}^\infty \in \mathcal{WN}_\xi(Y)$. From this it follows that $\mathfrak{G}_{\xi, \zeta}$ is simply the class $\mathfrak{L}$ of all bounded, linear operators when $0\leqslant \xi\leqslant \zeta\leqslant \omega_1$. Therefore we will be interested in these classes only in the non-trivial case $0\leqslant \zeta<\xi\leqslant \omega_1$. The classes $(\mathfrak{G}_{\xi, \zeta})_{0\leqslant \zeta<\xi\leqslant \omega_1}$ are closed, distinct, injective, two-sided ideals which contain all compact operators \cite{Causey1} and each of which contains the class $\mathfrak{V}$.  For $0\leqslant \zeta, \xi\leqslant \omega_1$, we let $\mathfrak{M}_{\xi, \zeta}$ denote the class of all operators $A:X\to Y$ such that for each $(x_n)_{n=1}^\infty\in \mathcal{WN}_\xi(X)$ and $(y^*_n)_{n=1}^\infty\in \mathcal{WN}_\zeta(Y^*)$, $\lim_n y^*_n(Ax_n)=0$.     Then $\mathfrak{DP}=\mathfrak{M}_{\omega_1, \omega_1}$.  Furthermore, $\mathfrak{M}_{\xi, \zeta}=\mathfrak{L}$ if $\xi=0$ or $\zeta=0$, so we will restrict our attention to the cases $1\leqslant \xi, \zeta\leqslant \omega_1$.   The classes $(\mathfrak{M}_{\xi, \zeta})_{1\leqslant \xi, \zeta\leqslant \omega_1}$ are closed, distinct, non-injective, two-sided ideals which contain all compact operators \cite{Causey1}.     One benefit of defining and studying such classes is that results which fail for a set which is too complex may have (sometimes quantitatively weaker) positive results when we restrict our attention to sets with lower complexity. Results of this type using descriptive set theory can be found in \cite{BC2} and \cite{BF}.  To that end, we show that when restricting to countable ordinals, we obtain strictly lower complexity for the classes in the coding of operators between separable Banach spaces.  We also compute complexity of the associated space ideals in the coding $\textbf{SB}$ of separable Banach spaces, which complements recent computations of Kurka of the classes of separable Schur spaces and separable spaces with the Dunford-Pettis property. Kurka's results are the spatial versions of items $(iii)$ and $(vi)$ of the following theorem.

\begin{theorem} \begin{enumerate}[(i)]\item For $0\leqslant \zeta<\xi<\omega_1$, the class $\mathfrak{G}_{\xi, \zeta}\cap \mathcal{L}$ is $\Pi_1^1$-complete and therefore non-Borel in the coding $\mathcal{L}$ of operators between separable Banach spaces. \item For $0\leqslant \zeta< \xi<\omega_1$, the class $\textsf{\emph{G}}_{\xi, \zeta}\cap \textbf{\emph{SB}}$ of spaces $X$ such that $I_X\in \mathfrak{G}_{\xi, \zeta}$ is $\Pi_1^1$-complete and therefore non-Borel in the coding $\textbf{\emph{SB}}$ of separable Banach spaces. \item For each $0\leqslant \zeta<\omega_1$, the class $\mathfrak{G}_{\omega_1, \zeta}\cap \mathcal{L}$ is $\Pi_2^1$-complete and therefore not $\Sigma_2^1$ in $\mathcal{L}$. \item For $1\leqslant \zeta, \xi<\omega_1$, the class $\mathfrak{M}_{\xi, \zeta}\cap \mathcal{L}$ is $\Pi_1^1$-complete and therefore non-Borel in the coding $\mathcal{L}$ of operators between separable Banach spaces. \item For $1\leqslant \zeta, \xi<\omega_1$, the class $\textsf{\emph{M}}_{\xi, \zeta}\cap \textbf{\emph{SB}}$ of spaces $X$ such that $I_X\in \mathfrak{M}_{\xi, \zeta}$ is $\Pi_1^1$-complete and therefore non-Borel in the coding $\textbf{\emph{SB}}$ of separable Banach spaces. \item For each $1\leqslant \zeta, \xi\leqslant \omega_1$ with $\max\{\xi, \zeta\}=\omega_1$, the class $\mathfrak{M}_{\omega_1, \zeta}\cap \mathcal{L}$ is $\Pi_2^1$-complete and therefore not $\Sigma_2^1$ in $\mathcal{L}$. \end{enumerate}

\end{theorem}

We also investigate the classes above by searching for one or a class of universal factoring operators for the complement of the ideal.  If $\mathfrak{I}$ is an ideal and $U:F\to G$ is a member of the complement $\complement \mathfrak{I}$ which factors through another operator $A:X\to Y$, then $A\in \complement \mathfrak{I}$.  This motivates a search for an easily understood class $\mathfrak{U}\subset \complement \mathfrak{I}$ such that for each $A:X\to Y\in \complement \mathfrak{I}$, there exists $U:F\to G\in \mathfrak{U}$ which factors through $A$.    The best result of this type would be for $\mathfrak{U}$ to be a singleton.  One notable of such results is the universal non-weakly compact operator $\Sigma:\ell_1\to \ell_\infty$ of Lindenstrauss and Pe\l cz\'{n}ski \cite{LP} which takes the $n^{th}$ member of the canonical $\ell_1$ basis to the sequence $(1, 1, \ldots, 1, 0, 0, \ldots)$, where $1$ appears $n$ times.  Another example is Johnson's universal non-compact operator $J:\ell_1\to \ell_\infty$ \cite{Johnson} which takes the canonical $\ell_1$ basis to the canonical $c_0$ basis.   A simple, universal class for the complement of an ideal can provide a route to investigating that ideal. For example, Bourgain's result \cite{Bourgain} that the binary trees of arbitrary, finite height embed with uniformly bounded distortion into any non-superreflexive Banach space uses the fact that the universal non-super-weakly compact operator factors through the identity of such a space. A generalization of this argument was used in \cite{CD} to prove  the analogous operator version of Bourgain's spatial result.   In certain instances, one can show that no universal operator exists for a given class (see, for example, \cite{GJ} and \cite{Oikhberg}), or that the existence of a  ``nice'' class of universal factoring operators is impossible (see \cite{BC}, where it was shown by descriptive set theoretic considerations that no Borel subset of $\mathcal{L}$ can be a universal factoring class for $\mathfrak{V}$).   As we have quantified classes which depend on ordinals parameters, one can ask for weaker conclusions by, for example,  searching for a ``nice'' subset $\mathfrak{U}$ of $\complement \mathfrak{G}_{\xi+1, \zeta}$    such that each member $A:X\to Y$ of $\complement \mathfrak{G}_{\xi, \zeta}$ factors of member of $\mathfrak{U}$.       This complements the negative result of Girardi and Johnson and offers another example of the aforementioned theme within descriptive set theory: Given an ordinal quantification on some class, restricting our attention to subsets whose ordinal quantification does not exceed some fixed, countable bound $\xi$ yields classes for which positive results hold, while the analogous results fail if we consider the entire class without a countable bound.    Our negative and positive results regarding universal classes are summarized in the following theorem. 

For the following theorem, if $F$ is a Banach space with basis $(f_i)_{i=1}^\infty$, and if $M=\{m_1<m_2<\ldots\}$ is an infinite subset of $\nn$,  we let $F_M$ denote the closed span in $F$ of the subsequence $(f_{m_i})_{i=1}^\infty$ of $(f_i)_{i=1}^\infty$. 

\begin{theorem} \begin{enumerate}[(i)]\item For $0\leqslant \zeta<\xi\leqslant \omega_1$,  $\complement\mathfrak{G}_{\xi, \zeta}$ does not admit a universal operator. \item For any $0\leqslant \zeta<\xi<\omega_1$, there exist a Banach space $F$ with basis $(f_i)_{i=1}^\infty$ and an operator $U:F\to \ell_\infty$ such that for each $\zeta<\beta<\xi$, each  subsequence $(f_{m_i})_{i=1}^\infty$ of the basis $(f_i)_{i=1}^\infty$ and $A:X\to Y\in \complement \mathfrak{G}_{\beta, \zeta}$, $U|_{F_M}\in \complement \mathfrak{G}_{\xi, \zeta}$ and $U|_{F_M}$ factors through $A$. \item For each $1\leqslant \xi, \zeta\leqslant \omega_1$, the class $\mathfrak{M}_{\xi, \zeta}$ does not admit a universal factoring operator. \end{enumerate}

\end{theorem}

We remark that, as $\complement\mathfrak{G}_{\beta, \zeta}\subset \complement \mathfrak{G}_{\xi, \zeta}\subset \complement \mathfrak{G}_{\alpha, \zeta}$ whenever $\beta<\xi<\alpha\leqslant \omega_1$, item $(ii)$ is quantitatively the strongest possible result in light of the negative result of $(i)$.   That is, for $0\leqslant \zeta<\beta<\omega_1$, we exhibit a fairly simple class $\mathfrak{U}$ of operators in $\complement\mathfrak{G}_{\beta+1, \zeta}$ such that each member of $\complement\mathfrak{G}_{\beta, \zeta}$ factors a member of $\mathfrak{U}$.

\section{Definitions}

Throughout, for a subset $M$ of $\nn$, we let $[M]$ (resp. $[M]^{<\nn}$) denote the set of all infinite (resp. finite) subsets of $M$. Throughout, we will denote sets as sets as well as strictly increasing sequences in the natural way.     We let $E<F$ denote the relation that either $E=\varnothing$, $F=\varnothing$, or $\max E<\min F$.       We topologize $\{0, 1\}^\nn$ with the product topology and endow the power set $2^\nn$ of $\nn$ with the topology making the map $2^\nn \ni E\leftrightarrow 1_E\in \{0,1\}^\nn$ a homeomorphism.    Given two members $(m_i)_{i=1}^k, (n_i)_{i=1}^k\in [\nn]^{<\nn}$, we say $(n_i)_{i=1}^k$ is a \emph{spread} of $(m_i)_{i=1}^k$ if $m_i\leqslant n_i $ for all $1\leqslant i\leqslant k$.    We say a subset $\mathcal{F}\subset [\nn]^{<\nn}$ is \emph{spreading} if it contains all spreads of its members. We say $\mathcal{F}\subset [\nn]^{<\nn}$ is \emph{hereditary} if it contains all subsets of its members. We say $\mathcal{F}\subset [\nn]^{<\nn}$ is \emph{regular} if it is spreading, hereditary, and compact.

Given two non-empty, regular families $\mathcal{F}, \mathcal{G}$,  we let $$\mathcal{G}[\mathcal{F}]=\{\varnothing\}\cup \Bigl\{\bigcup_{i=1}^n E_i: \varnothing \neq E_i\in \mathcal{F}, E_1<\ldots <E_n, (\min E_i)_{i=1}^n\in \mathcal{G}\Bigr\}.$$   Let $$\mathcal{A}_n=\{E: |E|\leqslant n\}$$ and let $$\mathcal{S}=\{\varnothing\}\cup \{E: |E|\leqslant \min E\}.$$    We next recall the Schreier families, defined in \cite{AMT}.    We let $$\mathcal{S}_0=\mathcal{A}_1, $$ if $\mathcal{S}_\xi$ has been defined for $\xi<\omega_1$, $$\mathcal{S}_{\xi+1}= \mathcal{S}[\mathcal{S}_\xi],$$ and if $\xi<\omega_1$ is a limit ordinal, there exists a sequence $(\xi_n)_{n=1}^\infty$ such that $\xi_n\uparrow \xi$, $\mathcal{S}_{\xi_n+1}\subset \mathcal{S}_{\xi_{n+1}}$ for all $n\in\nn$, and $$\mathcal{S}_\xi=\{\varnothing\}\cup \{E: \varnothing\neq E\in \mathcal{S}_{\xi_{\min E}+1}\}=\{E: \exists n\leqslant E\in \mathcal{S}_{\xi_n+1}\}.$$   We note that the existence of such as sequence was discussed, for example, in \cite{C2}.

For a regular family $\mathcal{F}$, let us note that the set of isolated points of $\mathcal{F}$ is precisely the set of maximal (with respect to inclusion) members of $\mathcal{F}$.  Let us denote this set by $MAX(\mathcal{F})$.  Then we let $\mathcal{F}'=\mathcal{F}\setminus MAX(\mathcal{F})$. It is easy to see that $\mathcal{F}'$ is also regular.   Then the Cantor-Bendixson derivatives are given by $$\mathcal{F}^0=\mathcal{F},$$ $$\mathcal{F}^{\xi+1}=(\mathcal{F}^\xi)',$$ and if $\xi$ is a limit ordinal, $$\mathcal{F}^\xi=\bigcap_{\zeta<\xi}\mathcal{F}^\zeta.$$  We let $CB(\mathcal{F})$ be the minimum ordinal $\xi$ such that $\mathcal{F}^\xi=\varnothing$, noting that such a $\xi$ must exist.   Furthermore, we note that for a non-empty, regular family $\mathcal{F}$, $CB(\mathcal{F})$ must be a successor ordinal. For this reason, it is convenient to let $\iota(\mathcal{F})=CB(\mathcal{F})-1$ whenever $\mathcal{F}$ is a non-empty, regular family.   We next recall some important facts regarding these notions. A reference for these facts is \cite{C2}.    For what follows, for $N=(n_i)_{i=1}^\infty\in [\nn]$ and a regular family $\mathcal{F}$, we let $\mathcal{F}(N)=\{(n_i)_{i\in E}: E\in \mathcal{F}\}$.

\begin{proposition} Let $\mathcal{F}, \mathcal{G}$ be regular families. \begin{enumerate}[(i)]\item For every $\xi<\omega_1$, $\mathcal{S}_\xi$ is regular with $\iota(\mathcal{S}_\xi)=\omega^\xi$. \item For every $n\in\nn$, $\mathcal{A}_n$ is regular with $\iota(\mathcal{A}_n)=n$. \item The set $\mathcal{F}[\mathcal{G}]$ is regular and $\iota(\mathcal{F}[\mathcal{G}])=\iota(\mathcal{G})+\iota(\mathcal{F})$. \item There exists $N\in[\nn]$ such that $\mathcal{F}(N)\subset \mathcal{G}$ if and only if for every $M\in[\nn]$, there exists $N\in[M]$ such that $\mathcal{F}(N)\subset \mathcal{G}$ if and only if $CB(\mathcal{F})\leqslant CB(\mathcal{G})$. \item If $\mathcal{G}$ is regular and $(m_n)_{n=1}^\infty\in [\nn]$, then $\{E\in[\nn]^{<\nn}: (m_n)_{n\in E}\in \mathcal{G}\}$ is regular with the same Cantor-Bendixson index as $\mathcal{G}$. \item For any $0\leqslant \zeta\leqslant \xi<\omega_1$, there exists $k\in\nn$ such that for any $E\in \mathcal{S}_\zeta$ with $k\leqslant E$, $E\in \mathcal{S}_\xi$. \end{enumerate}\label{gra}\end{proposition}

Given a regular family $\mathcal{F}$, a Banach space $X$, and a sequence $(x_n)_{n=1}^\infty\subset X$, let us say $(x_n)_{n=1}^\infty$ is an $\ell_1^\mathcal{F}+$-\emph{spreading model} if $(x_n)_{n=1}^\infty$ is bounded and $$\inf\{\|x\|:F\in \mathcal{F}, x\in \text{co}(x_n:n\in F)\}>0.$$   We say $(x_n)_{n=1}^\infty$ is an $\ell_1^\mathcal{F}$-spreading model if $(x_n)_{n=1}^\infty$ is bounded and $$\inf\{\|x\|: F\in \mathcal{F}, x=\sum_{n\in F} a_nx_n, \sum_{n\in F}|a_n|=1\}>0.$$   If $\mathcal{F}=\mathcal{S}_\xi$, we write $\ell_1^\xi+$ (resp. $\ell_1^\xi$) in place of $\ell_1^{\mathcal{S}_\xi}+$ (resp. $\ell_1^{\mathcal{S}_\xi}$).    For $\xi<\omega_1$, we say the sequence $(x_n)_{n=1}^\infty$ is $\xi$-\emph{weakly null} if it has no subsequence which is an $\ell_1^\xi+$-spreading model. This implies weak nullity by the Mazur lemma.

If $\mathcal{F}$ is  a regular family containing all singletons, we define the norm $\|\cdot\|_\mathcal{F}$ on $c_{00}$ by $$\|x\|_\mathcal{F}=\sup\{\|Ex\|_{\ell_1}: E\in \mathcal{F}\}.$$  Here, $E\subset \nn$ also denotes the projection on $c_{00}$ given by $E\sum_{n=1}^\infty a_ne_n=\sum_{n\in E}a_ne_n$.     In the case that $\mathcal{F}=\mathcal{S}_\xi$, we write $\|\cdot\|_\xi$ in place of $\|\cdot\|_{\mathcal{S}_\xi}$.    These are the Schreier spaces.     We also define the \emph{mixed Schreier spaces}. For a null sequence $(\varpi_n)_{n=1}^\infty\subset (0,1]$ and a sequence $\mathcal{F}_1, \mathcal{F}_2, \ldots$ of regular families such that each $\mathcal{F}_n$ contains all singletons, the completion of $c_{00}$ with respect to the norm $$[x]=\sup \{\varpi_n \|Ex\|_{\ell_1}: n\in\nn, E\in \mathcal{F}_n\}.$$

\section{The non-existence of universal operators}

A persistent question regarding any class $\mathfrak{I}$ with the ideal property is whether or not there exists an operator $U:X\to Y$ lying in $\complement \mathfrak{I}$ which factors through every member of $\complement\mathfrak{I}$.  Important examples of such operators are the universal factoring non-weakly compact operator $\Sigma:\ell_1\to \ell_\infty$ which takes the $\ell_1$ basis to the summing basis of $c_0$ \cite{LP}, universal factoring non-super weakly compact operator $\Sigma_n:(\oplus_{n=1}^\infty \ell_1^n)_{\ell_1}\to (\oplus_{n=1}^\infty \ell_\infty^n)_{\ell_\infty}$ which takes the basis of $\ell_1^n$ to the summing basis of $\ell_\infty^n$, a universal $\ell_p$-singular operator, which is any isomorphic embedding of $\ell_p$ into $\ell_\infty$, and the universal factoring non-super $\ell_p$-singular operator $jP$, where $P:(\oplus_{n=1}^\infty \ell_p^n)_{\ell_1}\to (\oplus_{n=1}^\infty \ell_p^n)_{c_0}$ is the formal inclusion and $j:(\oplus_{n=1}^\infty \ell_p^n)_{c_0}\to \ell_\infty$ is an isomorphic embedding \cite{Oikhberg}. In this section, we will prove that none of our classes of interest admits a universal factoring operator. We begin with a technical piece for later use. 

\begin{lemma} Fix $0<\eta<\omega_1$ and suppose that $\mathcal{F}_1, \mathcal{F}_2, \ldots$ are regular families which contain all singletons and such that $CB(\mathcal{F}_n)<\omega^\eta$ for all $n\in\nn$.  Fix a sequence $(\varpi_n)_{n=1}^\infty\subset (0, 1]$ of  numbers converging to $0$.    Define $|\cdot|_n$ on $c_{00}$ by $$|x|_i=\sup\Bigl\{\sum_{i=1}^t \|I_i x\|_{\ell_2}:t\in \nn, I_1<\ldots <I_t, (\min I_i)_{i=1}^t\in \mathcal{F}_n\Bigr\}.$$   Define $[\cdot]$ on $c_{00}$ by $[x]=\sup_n \varpi_n|x|_n$ and let $Z$ be the completion of $c_{00}$ with respect to this norm.    Then the canonical basis of $Z$ is $\eta$-weakly null.

\label{btsu}
\end{lemma}

In the proof below, we make use of the repeated averages hierarchy, introduced in \cite{AMT}. As a precise definition of the repeated averages hierarchy is not necessary for the following proof, and the full definition of the hierarchy would be unnecessarily technical, we simply state here the essential facts about the repeated averages hierarchy needed for the following proof.  For each $\xi<\omega_1$, each $n\in\nn$, and $M\in[\nn]$, $\mathbb{S}^\xi_{M,n}=(\mathbb{S}^\xi_{M,n}(i))_{i=1}^\infty$ is a sequence of non-negative numbers such that $1=\sum_{i=1}^\infty \mathbb{S}^\xi_{M,n}(i)$ and $\{i: \mathbb{S}^\xi_{M,n}(i)\neq 0\}\in \mathcal{S}_\xi$.

\begin{proof}[Proof of Lemma \ref{btsu}] Suppose the result is not true.  Then there exist a subsequence $(e_{m_i})_{i=1}^\infty$  and $0<\ee<1$ such that $$\ee<\{[x]: F\in \mathcal{S}_\eta, x\in \text{co}(e_{m_n}: n\in F)\}.$$  First choose $k\in\nn$ such that $1/k<\ee^2/16$.     Choose $N=(n_i)_{i=1}^\infty\in[\nn]$ such that $\mathcal{S}_\eta[\mathcal{A}_k](N)\subset \mathcal{S}_\eta$. Such an $N$ exists by Proposition \ref{gra}$(iv)$, since $$CB(\mathcal{S}_\eta[\mathcal{A}_k])=k\omega^\eta+1=\omega^\eta+1=CB(\mathcal{S}_\eta).$$      Fix $G_1<G_2<\ldots$ with $|G_i|=k$ and define $x_i=\frac{1}{k}\sum_{j\in G_i} e_{m_{n_j}}$. Note that $\|x_i\|_{\ell_2}=1/k^{1/2}<\ee/4$ and $\|x_i\|_{\ell_1}= 1$ for all $i\in\nn$. Also, by our choice of $N$, it follows that $$\ee \leqslant \inf\{[x]: F\in \mathcal{S}_\eta, x\in \text{co}(x_i: i\in F)\}.$$   Fix $m\in\nn$ such that $\varpi_m<\ee/2$.   Let $\mathcal{F}=\cup_{i=1}^m \mathcal{F}_i$ and note that $CB(\mathcal{F})= \max_{1\leqslant i\leqslant m}CB(\mathcal{F}_i)<\omega^\eta$.  For each $i\in\nn$, let $s_i=\max \text{supp}(x_i)$ and let $$\mathcal{G}=\{E: (s_i)_{i\in E}\in \mathcal{F}\},$$ so $CB(\mathcal{G})=CB(\mathcal{F})<\omega^\eta$ by Proposition \ref{gra}$(v)$.     By \cite[Lemma $4.3$]{CN}, there exists $P\in [\nn]$ such that $$\sup\{\mathbb{S}^\eta_{Q,1}(A): A\in \mathcal{G}, Q\in [P]\}<\ee/4.$$  Let $$x=\sum_{i=1}^\infty \mathbb{S}^\eta_{P,1}(i)x_i\in \text{co}(x_i: i\in\text{supp}(\mathbb{S}^\eta_{P,1})).$$  Since $\text{supp}(\mathbb{S}^\eta_{P,1})\in \mathcal{S}_\eta$, $[x]\geqslant \ee$.    Now fix $n\in\nn$.    If $n>m$, $$\varpi_n|x|_n \leqslant \varpi_n\|x\|_{\ell_1} <\ee/2.$$   If $1\leqslant n\leqslant m$, fix $I_1<\ldots <I_t$ such that $(\min I_i)_{i=1}^t\in \mathcal{F}_n$.    Let $A$ denote the set of those $i\in \text{supp}(\mathbb{S}^\eta_{N,1})$ such that $I_jx_i\neq 0$ for at least two values of $j$, and let $B=\text{supp}(\mathbb{S}^\eta_{N,1})\setminus A$.   For each $i\in A$, let $j_i$ be the minimum $j\in \{1, \ldots, t\}$ such that $I_jx_i\neq 0$ and note that $A\ni i\mapsto j_i$ is an injection of $A$ into $\{1, \ldots, t\}$. Moreover,  $(s_i)_{i\in A}$ is a spread of $(\min I_{j_i})_{i\in A}\subset (\min I_j)_{j=1}^t\in \mathcal{F}_n$, whence $(s_i)_{i\in A}\in \mathcal{F}$ and $A\in \mathcal{G}$.    Then \begin{align*} \varpi_n\sum_{j=1}^t\|I_j x\|_{\ell_2} & \leqslant  \sum_{i\in A} \mathbb{S}^\eta_{P,1}(i)\sum_{j=1}^\infty \|I_j x_i\|_{\ell_1} + \sum_{i\in B} \mathbb{S}^\eta_{P,1}(i)\|x_i\|_{\ell_2} \\ & \leqslant \sum_{i\in A} \mathbb{S}^\eta_{P,1}(i)\|x_i\|_{\ell_1} + \sum_{i\in B}\mathbb{S}^\eta_{P,1}(i)k^{-1/2}  \\ & \leqslant \mathbb{S}^\eta_{P,1}(A)+ k^{-1/2} \leqslant \ee/2. \end{align*}   Since this holds for any $I_1<\ldots <I_t$ with $(\min I_i)_{i=1}^t\in \mathcal{F}_n$, it follows that $\varpi_n|x|_n\leqslant \ee/2$.   Therefore we have shown that $$[x]=\sup_n \varpi_n[x]_n \leqslant \ee/2,$$ a contradiction.

\end{proof}

 We recall that for a sequence $(x_n)_{n=1}^\infty$ in the Banach space $X$ and $\delta>0$, $$\mathfrak{F}_\delta((x_n)_{n=1}^\infty)=\{E\in [\nn]^{<\nn}: (\exists x^*\in B_{X^*})(\forall n\in E)(\text{Re\ }x^*(x_n)\geqslant \delta)\}.$$  We will use the following fact.

\begin{lemma}\cite[Lemma $3.12$]{CN} For a Banach space $X$, $0<\eta<\omega_1$, and an $\eta$-weakly null sequence $(x_i)_{i=1}^\infty \subset X$, for every $\delta>0$ and $M\in[\nn]$, there exists $N\in [M]$ such that $$CB(\mathfrak{F}_\delta((x_i)_{i=1}^\infty)\cap [N]^{<\nn})<\omega^\eta.$$

\label{cb}
\end{lemma}

We are now ready to prove the non-existence of a universal operator for $\complement \mathfrak{M}_{\xi, \zeta}$. 

\begin{theorem} For $1\leqslant \zeta, \xi\leqslant \omega_1$, $\complement \mathfrak{M}_{\xi,\zeta}$ does not admit a universal factoring operator. 

\label{nu1}

\end{theorem}

\begin{proof} Seeking a contradiction, assume that $U:X\to Y$ is a universal factoring operator for $\complement \mathfrak{M}_{\xi, \zeta}$.  This means there exists a sequence $(x_n)_{n=1}^\infty \subset X$ which is $\xi$-weakly null and such that $\inf_n \|Ux_n\|>0$.   If $\xi<\omega_1$, let $\eta=\xi$.   If $\xi=\omega_1$, fix $\eta<\omega_1$ such that $(x_i)_{i=1}^\infty$ is $\eta$-weakly null. Note that in either case, $0<\eta<\omega_1$.       By Lemma \ref{cb}, we may select $M_1\supset M_2\supset \ldots$ such that for each $n\in\nn$, $CB(\mathfrak{F}_{3^{-n}}((x_i)_{i=1}^\infty)\cap [M_n]^{<\nn})<\omega^\eta$.    For each $n\in \nn$, we may fix $\nu_n<\eta$ and $k_n\in\nn$ such that $CB(\mathfrak{F}_{3^{-n}}((x_i)_{i=1}^\infty)\cap [M_n]^{<\nn})<\omega^{\nu_n}k_n$.   Now fix $m_1<m_2<\ldots$, $m_n\in M_n$, and let $M=(m_n)_{n=1}^\infty$.   First note that for any $n\in\nn$ and  $L\in [M]$, there exists $N\in [L]$ such that $CB(\mathfrak{F}_{3^{-n}}((x_i)_{i=1}^\infty)\cap [N]^{<\nn})<\omega^{\nu_n}k_n$. Indeed, any $N\in [L]$ which is also a subset of the tail set $(m_i)_{i=n}^\infty$ of $M$ has this property.  Let $\mathcal{F}_n=\mathcal{A}_{k_n}[\mathcal{S}_{\nu_n}]$, which has Cantor-Bendixson index $\omega^{\nu_n}k_n<\omega^\eta$.  Let $Z$ be the space from Lemma \ref{btsu} with $\varpi_n=2^{-n}$. More precisely, $$[x]= \sup\{2^{-n} \sum_{i=1}^t \|I_ix\|_{\ell_2}: n\in\nn, I_1<\ldots <I_t, (\min I_i)_{i=1}^t\in \mathcal{A}_{k_n}[\mathcal{S}_{\nu_n}]\}.$$    By Lemma \ref{btsu}, the basis of $Z$ is $\eta$-weakly null.    The basis of this space is also $\xi$-weakly null, since $\eta\leqslant \xi$.    Let $I:Z\to \ell_2$ be the formal inclusion. Since the canonical basis of $\ell_2^*$ is $1$-weakly null and $e^*_n(Ie_n)=1$ for all $n\in\nn$, $I\in \complement \mathfrak{M}_{\xi, 1}$.    This means there exist $R:X\to Z$ and $L:\ell_2\to Y$ such that $U=LIR$.  Since $(Rx_i)_{i=1}^\infty$ and $(IRx_i)_{i=1}^\infty$ are weakly null and seminormalized in $Z$ and $\ell_2$, respectively, by a standard perturbation argument, we may fix $N=(n_i)_{i=1}^\infty\in[\nn]$ and a block sequence $(z_i)_{i=1}^\infty$ with respect to the $c_{00}$ basis such that $\ee:=\inf_i \|z_i\|_{\ell_2}>0$ and for all $(a_i)_{i=1}^\infty\in c_{00}$, $$[\sum_{i=1}^n a_iz_i]\leqslant 2[\sum_{i=1}^\infty a_i Rx_{n_i}].$$     Fix $n\in\nn$ such that $2^n/3^n<2\|R\|/\ee$.     By our remark above, by replacing $N$ with an infinite subset thereof, we may assume $$CB(\mathfrak{F}_{3^{-n}}((x_i)_{i=1}^\infty)\cap [N]^{<\nn})<\omega^{\nu_n}k_n.$$      Note that for any $F\in \mathcal{A}_{k_n}[\mathcal{S}_{\nu_n}]$ and non-negative scalars $(a_i)_{i\in F}$ summing to $1$, if $I_i=\text{supp}(z_i)$, then since $(\min \text{supp}(z_i))_{i\in F}$ is a spread of $F$,  $$[\sum_{i\in F}a_iz_i]\geqslant 2^{-n}\sum_{j\in F}\|I_j\sum_{i\in F}a_iz_i\|_{\ell_2} \geqslant 2^{-n} \ee.$$

Next let us note that by the geometric Hahn-Banach theorem, for any sequence $(y_i)_{i=1}^\infty$ and $\delta>0$, $F\in \mathfrak{F}_\delta((y_i)_{i=1}^\infty)$ if and only if $\min \{\|y\|: y\in \text{co}(y_i: i\in F)\}\geqslant \delta$.  By the last inequality from the previous paragraph, it follows that $CB(\mathfrak{F}_{2^{-n}\ee}((z_i)_{i=1}^\infty))\geqslant CB(\mathcal{A}_{k_n}[\mathcal{S}_{\nu_n}])=\omega^{\nu_n}k_n+1$.     Now for any finite subset $F$ of $\nn$, $$\min \{\|z\|: z\in \text{co}(z_i:i\in F)\} \leqslant 2\min\{\|Rx\|: x\in \text{co}(x_{n_i}: i\in F)\}\leqslant 2\|R\|\min \{\|x\|: x\in \text{co}(x_{n_i}: i\in F)\}.$$  From this it follows that for any $\delta>0$, $\mathfrak{F}_\delta((z_i)_{i=1}^\infty)\subset \mathfrak{F}_{2\|R\|\delta}((x_{n_i})_{i=1}^\infty)$ and $$CB(\mathfrak{F}_\delta((z_i)_{i=1}^\infty))\leqslant CB(\mathfrak{F}_{2\|R\|\delta}((x_{n_i})_{i=1}^\infty)) \leqslant CB(\mathfrak{F}_{2\|R\|\delta}((x_i)_{i=1}^\infty)\cap [N]^{<\nn}).$$   Applying this with $\delta= 2^{-n}\ee$ and noting that $2\|R\|\delta>3^{-n}$,\begin{align*} \omega^{\nu_n}k_n & <CB(\mathfrak{F}_{2^{-n}\ee/2}((z_i)_{i=1}^\infty)) \leqslant CB(\mathfrak{F}_{2\|R\|2^{-n}\ee/2}((x_i)_{i=1}^\infty)\cap [N]^{<\nn}) \\ & \leqslant CB(\mathfrak{F}_{3^{-n}}((x_i)_{i=1}^\infty)\cap [N]^{<\nn})<\omega^{\nu_n}k_n.\end{align*}  This contradiction finishes the proof.

\end{proof}

For the classes $\mathfrak{G}_{\xi, \zeta}$, we will prove a slightly stronger non-existence result, followed by a parallel positive result.   Let $F,G$ be Banach spaces with bases $(f_i)_{i=1}^\infty$, $(g_i)_{i=1}^\infty$ such that the formal inclusion $U:F\to G$ is well-defined.  For an infinite subset $M\in[\nn]$, let $F_M$ (resp. $G_M$) denote the closed span of $(f_{m_i})_{i=1}^\infty$ in $F$ (resp. $(g_{m_i})_{i=1}^\infty$ in $G$).   Let $U_M:F_M\to G_M$ be the restriction of the formal inclusion $U$ to $F_M$. Given an ideal $\mathfrak{I}$, let us say that $U$ is \emph{subsequentially universal} for $\complement \mathfrak{I}$ if  for any $A:X\to Y\in \complement \mathfrak{I}$ and any $L\in[\nn]$, there exist $M\in [L]$ and subspaces $X_0$, $Y_0$ of $X$ and $Y$, respectively, such that $A(X_0)\subset Y_0$ and $U_M$ factors through $A|_{X_0}:X_0\to Y_0$.

\begin{rem}\upshape An operator $U:F\to G$ being subsequentially universal for the class $\complement\mathfrak{I}$ provides a potentially small, easy to understand collection (formal inclusions between subsequences of fixed bases) which can be used to study the class $\mathfrak{I}$.  Furthermore, the definition not only requires that we can factor formal inclusions of these subsequences through members of $\complement \mathfrak{I}$, but that the subsequences of this type are fairly abundant.

\end{rem}

\begin{rem}\upshape If $\mathfrak{I}$ is injective (which our classes $\mathfrak{G}_{\xi, \zeta}$ are) and $U:F\to G$ is subsequentially universal for $\complement \mathfrak{I}$, first fix an isometric embedding $j:G\to \ell_\infty$. Then the conclusion that $U_M$ factors through a restriction $A|_{X_0}:X_0\to Y_0$ of $A$ together with injectivity imply the existence of a factorization of $jU_M$ through $A$.

\end{rem}

\begin{theorem} \begin{enumerate}[(i)]\item For $0\leqslant \zeta<\xi\leqslant \omega_1$, there do not exist basic sequences $(f_i)_{i=1}^\infty$ and $(g_i)_{i=1}^\infty$ such that $(f_i)_{i=1}^\infty$ is $\xi$-weakly null, $(g_i)_{i=1}^\infty$ is not $\zeta$-weakly null, and the formal identity $U:[f_i:i\in\nn]\to [g_i:i\in\nn]$ is well-defined and subsequentially universal for $\complement\mathfrak{G}_{\xi, \zeta}$. 
\item For any $0\leqslant \zeta<\xi\leqslant \omega_1$, there exists a formal identity operator $I$ between mixed Schreier spaces which lies in $\complement \mathfrak{G}_{\xi, \zeta}$ such that for each $\zeta<\beta<\xi$, $I$ is subsequentially universal for $\complement \mathfrak{G}_{\beta, \zeta}$.   \end{enumerate}

\label{big show}
\end{theorem}

Given item $(i)$ of Theorem \ref{big show}, item $(ii)$ Theorem \ref{big show} is the best possible quantitative weakening in the search for universal factoring operators.

We will need the following consequence of Lemma \ref{btsu}.

\begin{proposition} Suppose $0<\eta<\omega_1$ and $\mathcal{F}_1, \mathcal{F}_2, \ldots$ are regular families containing all singletons and such that for each $j\in \nn$, $CB(\mathcal{F}_j)<\omega^\eta$. Fix a positive sequence of numbers $(\varpi_n)_{n=1}^\infty$ converging to zero.    Let $Z$ be the completion of $c_{00}$ with respect to the norm $$[x]_0= \sup\Bigl\{\varpi_j\|Ex\|_{\ell_1}: j\in \nn, E\in\mathcal{F}_j\}.$$  Then the basis of $Z_0$ is $\eta$-weakly null.

\label{sop}
\end{proposition}

\begin{proof} Fix $j\in\nn$ and $E\in \mathcal{F}_j$.  Write $E=(n_i)_{i=1}^t$ and let $I_i=(n_i)$.    Then $I_1<\ldots <I_t$ and $(\min I_i)_{i=1}^t\in \mathcal{F}_j$.    Then if $Z$ is the space from Lemma \ref{btsu}, $$[x]\geqslant \varpi_j\|Ex\|_{\ell_1}.$$  From this it follows that the formal inclusion $I:Z\to Z_0$ is bounded with norm $1$. Since the canonical $c_{00}$ basis is $\eta$-weakly null in $Z$, its image is $\eta$-weakly null in $Z_0$.

\end{proof}

\begin{proof}[Proof of Theorem \ref{big show}]$(i)$ Seeking a contradiction, assume $F$ is the closed span of a $\xi$-weakly null, basic sequence $(f_i)_{i=1}^\infty$, $G$ is the closed span of a  basic, $\ell_1^\zeta+$-spreading model $(g_i)_{i=1}^\infty$, and the linear extension of the map taking $f_i$ to $g_i$ extends to a continuous linear operator $U:F\to G$ which is subsequentially universal for $\complement \mathfrak{G}_{\xi, \zeta}$.    If $\xi<\omega_1$, let $\eta=\xi$. If $\xi=\omega_1$, let $\eta<\omega_1$ be such that $(f_i)_{i=1}^\infty$ is $\eta$-weakly null.    Note that in either case, $0<\eta<\omega_1$.    As in the proof of Theorem \ref{nu1}, we may recursively select $M_1\supset M_2\supset \ldots$, $\nu_n<\eta$, $k_n\in\nn$ such that for all $n\in\nn$,  $$CB(\mathfrak{F}_{3^{-n}}((f_i)_{i=1}^\infty)\cap [M_n]^{<\nn})<\omega^{\nu_n}k_n.$$  Let us note that since $(f_i)_{i=1}^\infty$ is $\eta$-weakly null and $(g_i)_{i=1}^\infty=(Uf_i)_{i=1}^\infty$ is not, $\zeta<\eta$.    Therefore by replacing $\nu_n$ with $\zeta$ for any $n$ such that $\nu_n<\zeta$, we may assume that $\nu_n\geqslant \zeta$ for all $n\in\nn$.      Let $\varrho_n=\nu_n-\zeta$. That is, $\varrho_n$ is the unique ordinal such that $\zeta+\varrho_n=\nu_n$.     For each $n\in\nn$, let $$\mathcal{F}_n= \mathcal{A}_{k_n}[\mathcal{S}_{\varrho_n}[\mathcal{S}_\zeta]].$$ Note that for each $n\in\nn$, $CB(\mathcal{F}_n)=\omega^\zeta \omega^{\varrho_n} k_n+1=\omega^{\nu_n}k_n+1<\omega^\eta$. Note that $\mathcal{F}_n\supset \mathcal{S}_\zeta$ for all $n\in\nn$.    Let $Z$ be the completion of $c_{00}$ with respect to the mixed Schreier norm $$[x]=\sup \{2^{-n}\|Ex\|_{\ell_1}: n\in\nn, E\in\mathcal{F}_n\}.$$ Let $I:Z\to X_\zeta$ be the formal inclusion, which has norm $2$. Moreover, by Proposition \ref{sop}, the basis of $Z$ is $\eta$-weakly null.  But the canonical $X_\zeta$ basis is not $\zeta$-weakly null, so $I\in \complement \mathfrak{G}_{\eta, \zeta}\subset \complement \mathfrak{G}_{\xi, \zeta}$.

Fix $m_1<m_2<\ldots$, $m_n\in M_n$, and let $M=(m_n)_{n=1}^\infty$.       By the definition of subsequentially universal, there exists $P\in [M]$ such that the restriction $U_P:F_P\to G_P$ factors through some restriction of $I$. Fix $V_0, W_0$ and $R:F_P\to V_0$, $L:W_0\to G_P$ such that $I(U_0)\subset V_0$ and $LIR=U_P$.  Since $(LIRf_n)_{n\in P}$ is an $\ell_1^\zeta+$-spreading model, so is $(IRf_n)_{n\in P}$, and $(IRf_n)_{n\in P}$ has no $\zeta$-weakly null subsequence.     By standard perturbation arguments, we may find $N=(n_i)_{i=1}^\infty\in [P]$ and  a block sequence $(z_i)_{i=1}^\infty$ with respect to the $c_{00}$ basis such that $(Iz_i)_{i=1}^\infty$ is an $\ell_1^\zeta+$-spreading model in $X_\zeta$ and for all $(a_i)_{i=1}^\infty$, $$[\sum_{i=1}^\infty a_i z_i]\leqslant 2[\sum_{i=1}^\infty a_i Rf_{n_i}].$$      We must consider two cases.  If $\zeta=0$, we fix $0<\ee<\inf_i \|z_i\|_\zeta=\inf_i \|z_i\|_{c_0}$.     If $\zeta>0$, let $\gamma=\max\{\omega^\alpha: \omega^\alpha \leqslant \zeta\}$.  In the $\zeta>0$ case, by \cite[Theorem $2.14$]{Causey1}, there exists $\beta<\gamma$ such that $\lim\sup \|z_i\|_\beta>0$. In this case, we may pass to a subsequence of $N$, relabel, and assume there exists $\ee>0$ such that $\ee<\inf_i \|z_i\|_\beta$. This is because if no such $\ee$ and $\beta$ exist, $(Iz_i)_{i=1}^\infty$ is $\zeta$-weakly null in $X_\zeta$.        Now in either case, fix $n\in\nn$ so large that $2^n/3^n<2\|R\|/\ee$.    By passing to a subset of $N$ and relabeling  once more, we may assume that $$CB(\mathfrak{F}_{3^{-n}}((f_{n_i})_{i=1}^\infty))\leqslant CB(\mathfrak{F}_{3^{-n}}((f_i)_{i=1}^\infty)\cap [N]^{<\nn})< \omega^{\nu_n}k_n.$$      As in the proof of Theorem \ref{nu1}, we will show that $CB(\mathfrak{F}_{2^{-n}\ee}((z_i)_{i=1}^\infty))>\omega^{\nu_n}k_n$, and this contradiction will finish $(i)$.    In the $\zeta=0$ case, for each $i\in \nn$, we fix a singleton $E_i=(s_i)\in \text{supp}(z_i)$ such that $\|E_iz_i\|_{\ell_1}>\ee$.     In the $0<\zeta$ case, we fix $E_i\in \mathcal{S}_\beta$ such that $\|E_iz_i\|_{\ell_1}>\ee$.     In either case, there exists $T=(t_i)_{i=1}^\infty\in [\nn]$ such that for any $G\in \mathcal{S}_\zeta$, $\cup_{i\in G}E_{t_i}\in \mathcal{S}_\zeta$.    In the $\zeta=0$ case, we may take $T=\nn$, since $G$ and $E_i$ are singletons.  For the $\zeta>0$ case, we appeal to \cite[Lemma $2.2$$(i)$]{Causey1} and the fact that $\beta+\zeta=\zeta$ by properties of the additively indecomposable ordinal $\gamma$.   In the $\zeta=0$ case, $\varrho_n=\nu_n$ and $\mathcal{F}_n=\mathcal{A}_{k_n}[\mathcal{S}_{\nu_n}]$.  Then for any $F\in \mathcal{F}_n$ and non-negative scalars $(a_i)_{i\in F}$ summing to $1$, $$[\sum_{i\in F} a_i z_i]\geqslant 2^{-n}\sum_{j\in F}\|E_j \sum_{i\in F}a_iz_i\|_{\ell_1} \geqslant 2^{-n}\ee.$$   By another appeal to the geometric Hahn-Banach theorem, $\mathcal{A}_{k_n}[\mathcal{S}_{\nu_n}]\subset \mathfrak{F}_{2^{-n}\ee}((z_i)_{i=1}^\infty)$, which gives the required lower estimate on the Cantor-Bendixson index and finishes the $\zeta=0$ case of the proof.   Now assume $\zeta>0$ and let $T$ be as above.   Now fix $F\in \mathcal{A}_{k_n}[\mathcal{S}_{\varrho_n}[\mathcal{S}_\zeta]]$, which means we can write $$F=\bigcup_{i=1}^l G_i,$$ $G_1<\ldots <G_i$, $\varnothing\neq G_i\in \mathcal{S}_\zeta$, $(\min G_i)_{i=1}^l\in\mathcal{A}_{k_n}[\mathcal{S}_{\varrho_n}]$.    Now for each $1\leqslant i\leqslant l$, let $H_i=\cup_{j\in G_i} E_{t_j}$ and note that this set lies in $ \mathcal{S}_\zeta$ by our choice of $T$. Note that $\min H_i\geqslant \min G_i$, so $(\min H_i)_{i=1}^l$ is a spread of $(\min G_i)_{i=1}^l$, and therefore lies in $\mathcal{A}_{k_n}[\mathcal{S}_{\varrho_n}]$.    Therefore $$H:=\bigcup_{i=1}^l H_i\in \mathcal{A}_{k_n}[\mathcal{S}_{\varrho_n}[\mathcal{S}_\zeta]].$$  For any non-negative scalars $(a_i)_{i\in F}$ summing to $1$, $$[\sum_{i\in F} a_i z_{t_i}]\geqslant 2^{-n}\|H\sum_{i\in F}a_iz_{t_i}\|_{\ell_1} \geqslant 2^{-n}\ee.$$    One more appeal to the geometric Hahn-Banach theorem yields that $$\{(t_i)_{i\in E}: E\in \mathcal{A}_{k_n}[\mathcal{S}_{\varrho_n}[\mathcal{S}_\zeta]]\}\subset \mathfrak{F}_{2^{-n}\ee}((z_i)_{i=1}^\infty).$$   Since $$CB(\{(t_i)_{i\in E}: E\in \mathcal{A}_{k_n}[\mathcal{S}_{\varrho_n}[\mathcal{S}_\zeta]]\})=CB(\mathcal{A}_{k_n}[\mathcal{S}_{\varrho_n}[\mathcal{S}_\zeta]])=\omega^\zeta\omega^{\varrho_n}k_n+1=\omega^{\nu_n}k_n+1,$$  this gives the required lower estimate on the Cantor-Bendixson index and finishes $(i)$.

$(ii)$  We first consider the case in which $\xi$ is a successor, say $\xi=\eta+1$.     If $\xi=\zeta+1$, the conclusion is vacuous, as there are no $\beta$ with $\zeta<\beta<\xi$. Therefore we assume $\zeta<\eta$, so $0<\eta<\omega_1$.      Let $I:X_\eta\to X_\zeta$ be the formal inclusion, which is bounded by Proposition \ref{gra}.    Let $X,Y$ be Banach spaces and suppose $A:X\to Y\in \complement \mathfrak{G}_{\eta, \zeta}$.   Fix an $\eta$-weakly null sequence $(x_i)_{i=1}^\infty\subset B_X$ such that $(Ax_i)_{i=1}^\infty$ is basic, seminormalized and an $\ell_1^\zeta+$-spreading model.  By \cite[Theorem A]{AMT}, we may assume that $(x_i)_{i=1}^\infty$ and $(Ax_i)_{i=1}^\infty$ are convexly unconditional.  This means that for any $\delta>0$, there exists $C(\delta)>0$ such that for any $(a_i)_{i=1}^\infty$ with $\sum_{i=1}^\infty |a_i|\leqslant 1$ and $\|\sum_{i=1}^\infty a_ix_i\|\geqslant \delta$, then $\|\sum_{i=1}^\infty \lambda_i a_ix_i\|\geqslant C(\delta)$ for any scalars $(\lambda_i)_{i=1}^\infty$ with $|\lambda_i|=1$ for all $i\in\nn$.  A similar inequality holds for $(Ax_i)_{i=1}^\infty$.      By \cite[Theorem $1.10$]{AG}, we may assume $(y_i)_{i=1}^\infty$ is $\mathcal{S}_\zeta$-unconditional. This means that there exists a constant $a>0$ such that for any $(a_i)_{i=1}^\infty\in c_{00}$, $$\sup_{F\in \mathcal{S}_\zeta} a\|\sum_{i\in F}a_i Ax_i\|\leqslant \|\sum_{i=1}^\infty a_i Ax_i\|.$$   Since  $(Ax_i)_{i=1}^\infty$ is convexly unconditional and an $\ell_1^\zeta+$-spreading model,   $(Ax_i)_{i=1}^\infty$ is an $\ell_1^\zeta$-spreading model, which means there exists $b>0$ such that $$\inf\{\|\sum_{i\in F} a_i Ax_i\|: F\in \mathcal{S}_\zeta, \sum_{i\in F}|a_i|=1\}=b.$$   Then for any $(a_i)_{i=1}^\infty \in c_{00}$, $$\|\sum_{i=1}^\infty a_i Ax_i\|\geqslant a\sup \{\|\sum_{i\in F}a_i Ax_i\|: F\in \mathcal{S}_\zeta\} \geqslant ab \sup \{\sum_{i\in F}|a_i|: F\in \mathcal{S}_\zeta\}.$$  This yields that the formal inclusion $J:[Ax_i:i\in\nn]\to X_\zeta$ given by $JAx_i=e_i$ is well-defined and bounded.     Now for $\delta>0$, let $$\mathfrak{H}_\delta=\{E\in [\nn]^{<\nn}: (\exists x^*\in B_{X^*})(\forall n\in E)(|x^*(x_i)|\geqslant \delta)\}.$$  Since $(x_i)_{i=1}^\infty$ is $\eta$-weakly null, for every $\delta>0$ and $M\in[\nn]$, there exists $N\in[M]$ such that $$CB(\mathfrak{F}_\delta((x_i)_{i=1}^\infty)\cap [N]^{<\nn})<\omega^\eta.$$  By convex unconditionality, this implies that for every $\delta>0$ and $M\in [\nn]$, there exists $N\in [M]$ such that $CB(\mathfrak{H}_\delta\cap [N]^{<\nn})<\omega^\eta$.    Now let us fix $0<\vartheta<1$.  Let $L\in [\nn]$ be arbitrary and recursively select $L\supset M_1\supset M_2\supset \ldots$ such that for all $n\in\nn$, either $\mathfrak{H}_{\vartheta^n}\cap [M_n]^{<\nn}\subset \mathcal{S}_\eta$ or $\mathcal{S}_\eta\cap [M_n]^{<\nn}\subset \mathfrak{H}_{\vartheta^n}$. We may make these selections by \cite[Theorem $1.1$]{Gasparis}.     But our remark preceding the fixing of $\vartheta$ yields that the second option cannot hold, and $\mathfrak{H}_{\vartheta^n}\cap [M_n]^{<\nn}\subset \mathcal{S}_\eta$ for all $n\in\nn$.    Fix $m_1<m_2<\ldots$, $m_n\in M_n$, and let $M=(m_n)_{n=1}^\infty$.  Now fix $(a_i)_{i=1}^\infty\in c_{00}$ and $x^*\in B_{X^*}$ such that $$\|\sum_{i=1}^\infty a_ix_{m_i}\|=x^*(\sum_{i=1}^\infty a_i x_{m_i}).$$   For each $n\in\nn$, let $$I_n=\{i<n: |x^*(x_{m_i})|\in (\vartheta^n, \vartheta^{n-1}]\}$$ and $$J_n=\{i\geqslant n: |x^*(x_{m_i})|\in (\vartheta^n, \vartheta^{n-1}]\}.$$  For each $n\in\nn$, $$(m_i)_{i\in J_n}\in \mathfrak{H}_{\vartheta^n}\cap [M_n]^{<\nn}\subset \mathcal{S}_\eta,$$ so $$x^*(\sum_{i\in I_n\cup J_n} a_ix^*_{m_i}) \leqslant \vartheta^{n-1}\Bigl[\|(a_i)_{i=1}^\infty\|_\infty |I_n|+ \sum_{i\in J_n}|a_i|\Bigr] \leqslant n\vartheta^{n-1}\|\sum_{i=1}^\infty a_i e_{m_i}\|_\eta.$$   Therefore $$\|\sum_{i=1}^\infty a_ix_{m_i}\|\leqslant (1-\vartheta)^{-2}\|\sum_{i=1}^\infty a_ie_{m_i}\|_\eta.$$     Thus the maps taking $(e_{m_i})_{i=1}^\infty\subset X_\eta$ to $(x_{m_i})_{i=1}^\infty$ and $(Ax_{m_i})_{i=1}^\infty$ to $(e_{m_i})_{i=1}^\infty \subset X_\zeta$ are bounded.    Since $L\in[\nn]$ was arbitrary, this shows that $I:X_\eta\to X_\zeta$ is subsequentially universal for $\complement \mathfrak{G}_{\eta, \zeta}$.  Since for any $\zeta<\beta<\xi$, $\complement \mathfrak{G}_{\beta, \zeta}\subset \complement \mathfrak{G}_{\eta, \zeta}$, this completes the successor case.

Now suppose that $\xi$ is a limit ordinal.   Fix $\zeta<\xi_1<\xi_2<\ldots$ with $\xi_n\uparrow \xi$ and a null sequence $(\varpi_n)_{n=1}^\infty\subset (0,1]$ of positive numbers.  Let $Z$ be the completion of $c_{00}$ with respect to the norm $$[x]=\sup \{\varpi_n \|Ex\|_{\ell_1}: n\in\nn, E\in \mathcal{S}_{\xi_n}\}$$ and let $I:Z\to X_\zeta$ be the formal inlusion.      Suppose that for some $0\leqslant \zeta<\beta<\xi$, $A:X\to Y\in \complement \mathfrak{G}_{\beta, \zeta}$.    Arguing as in the successor case, we may select a sequence $(x_i)_{i=1}^\infty\subset B_X$ which is $\beta$-weakly null and such that the map taking $(Ax_i)_{i=1}^\infty $ to $(e_i)_{i=1}^\infty\subset X_\zeta$ is bounded.     Also, for $L\in[\nn]$, we may select $(m_i)_{i=1}^\infty\in [L]$ such that the map taking $(e_{m_i})_{i=1}^\infty\subset X_\beta$ to $(x_{m_i})_{i=1}^\infty$ is bounded.    Now if $n\in\nn$ is such that $\xi_n>\beta$, the formal inclusion of $X_{\xi_n}$ into $X_\beta$ is bounded by Proposition \ref{gra}, as is the map taking $(e_{m_i})_{i=1}^\infty \subset X_{\xi_n}$ to $(e_{m_i})_{i=1}^\infty$ into $X_\beta$.   Now the map taking $(e_{m_i})_{i=1}^\infty \subset Z$ to $(e_{m_i})_{i=1}^\infty\subset X_{\xi_n}$, and therefore the maps taking $(e_{m_i})_{i=1}^\infty\subset Z$ to $(x_{m_i})_{i=1}^\infty$ and $(Ax_{m_i})_{i=1}^\infty$ to $(e_{m_i})_{i=1}^\infty\subset X_\zeta$ are well-defined and bounded.      This yields the appropriate factorization of $I$ through a restriction of $A$ and gives the limit ordinal case.

\end{proof}

The following is implicitly contained in $(i)$ of the preceding proof. 

\begin{corollary} For $0\leqslant \zeta<\xi\leqslant \omega_1$, $\complement \mathfrak{G}_{\xi, \zeta}$ does not admit a universal factoring operator.

\end{corollary}

\begin{proof}[Sketch] If $U:F\to G\in \complement \mathfrak{G}_{\xi, \zeta}$ were universal for $\complement \mathfrak{G}_{\xi,\zeta}$, we first fix $(f_i)_{i=1}^\infty\subset X$ which is $\xi$-weakly null and such that $(Uf_i)_{i=1}^\infty=(g_i)_{i=1}^\infty$ is an $\ell_1^\zeta+$-spreading model.    Let $\eta=\xi$ if $\xi<\omega_1$ and otherwise let $\eta<\omega_1$ be such that $(f_i)_{i=1}^\infty$ is $\eta$-weakly null.     We fix $M_1\supset M_2\supset \ldots$, $\nu_n<\eta$, and $k_n\in\nn$ such that $$CB(\mathfrak{F}_{3^{-n}}((f_i)_{i=1}^\infty)\cap [M_n]^{<\nn})<\omega^{\nu_n}k_n.$$  We may assume $\nu_n\geqslant \zeta$ for all $n\in\nn$ and write $\nu_n=\zeta+\varrho_n$.    If $\zeta=0$, let $Z$ be the completion of $c_{00}$ with respect to the mixed Schreier norm $$[x]=\sup\{2^{-n}\|Ex\|_{\ell_1}: n\in\nn, E\in \mathcal{A}_{k_n}[\mathcal{S}_{\nu_n}]\}.$$  If $\zeta>0$, let $Z$ be the completion of $c_{00}$ with respect to the mixed Schreier norm $$[x]=\sup \{2^{-n}\|Ex\|_{\ell_1}: n\in\nn, E\in \mathcal{A}_{k_n}[\mathcal{S}_{\varrho_n}[\mathcal{S}_\zeta]]\}.$$   In either case, the formal inclusion $I:Z\to X_\zeta$ lies in $\complement \mathfrak{G}_{\xi, \zeta}$. If $U$ were to factor through $Z$ as $U=LIR$, then arguing as in the proof of Theorem \ref{big show}$(i)$, we would be able to find $N\in[M]$, $\ee>0$, $n$ such that $2^n/3^n<2\|R\|/\ee$, and a block sequence $(z_i)_{i=1}^\infty$ with respect to $c_{00}$ such that \begin{align*} \omega^{\nu_n}k_n & < CB(\mathfrak{F}_{2^{-n}\ee}((z_i)_{i=1}^\infty)) \leqslant CB(\mathfrak{F}_{3^{-n}}((f_i)_{i=1}^\infty)\cap [N]^{<\nn})<\omega^{\nu_n}k_n. \end{align*}

\end{proof}

\section{Codings of $\text{SB}$, $\mathcal{L}$, and dual spaces}

We first recall some facts and constructions from descriptive set theory. Two references for such facts are the books  \cite{Dodos} and \cite{Ke}.

The following fact is standard. However, since it will be used freely, we isolate it here. 

\begin{proposition} Let $X$ be a Polish space with topology $\tau$. Let $Y_n$ be a sequence of second countable topological spaces and $f_n:X\to Y_n$ a sequence of Borel functions. Then there exists a Polish topology $\tau'$ on $X$ finer than $\tau$ such that the Borel $\sigma$-algebras of $\tau$ and $\tau'$ coincide and for each $n\in\nn$, $f_n:(X, \tau')\to Y_n$ is continuous. 

\label{improv}
\end{proposition}

\begin{proof} For each $n\in\nn$, let $(U_{m,n})_{m=1}^\infty$ be a countable base for the topology of $Y_n$. Then by \cite[Lemma $13.3$, Page $82$]{Ke}, there exists a Polish topology $\tau'$ on $X$ finer than $\tau$, generating the same Borel $\sigma$-algebra as $\tau$,  such that each of the sets $f^{-1}_n(U_{m,n})$, $m,n\in\nn$, is clopen in $\tau'$.

\end{proof}

Let us recall that a subset $A$ of a Polish space $S$ is $\Sigma_1^1$ if there exist a Polish space $P$,  a Borel subset $B$ of $P$, and a Borel function $f:P\to S$ such that $f(B)=A$.    A subset $C$ of $S$ is $\Pi_1^1$ if  $S\setminus C$ is $\Sigma_1^1$.    We say a subset $A$ of a Polish space $S$ is $\Sigma_2^1$ if there exist a Polish space $P$, a $\Pi_1^1$ subset $B$ of $P$, and a Borel function $f:S\to P$ such that $f(B)=A$.  If for subset $A,B$ of Polish spaces $S,P$, respectively, and a Borel function $f:S\to P$, $f^{-1}(B)=A$, then we say $A$ is \emph{Borel reducible} to $B$.     Given $j\in \{1, 2\}$, we say a set $A\subset S$ is $\Pi^1_j$-\emph{hard} if for any Polish space $P$ and any $\Pi_j^1$ subset $B$ of $P$, $A$ is Borel reducible to $B$.  We say $A\subset S$ is $\Pi^1_j$-\emph{complete} if it is $\Pi_j^1$-hard and $\Pi_j^1$.

We let $C(2^\nn)$ be the space of continuous functions on the Cantor set.    We endow $F(C(2^\nn))$, the set of closed subsets of $C(2^\nn)$, with its Effros-Borel $\sigma$-algebra, and recall that this is a standard Borel space.   That is, there exists a Polish topology on $F(C(2^\nn))$ the Borel $\sigma$-algebra of which is the Effros-Borel $\sigma$-algebra.  By a result of Kuratowski and Ryll-Nardzewski \cite{KRN}, there exists a sequence $d_n:F(C(2^\nn))\setminus\{\varnothing\}\to C(2^\nn)$ such that each $d_n$ is Borel and for each $\varnothing\neq F\in F(C(2^\nn))$, $\{d_n(F): n\in\nn\}$ is a dense subset of $F$.   By standard techniques, we may assume that for each finite subset $F$ of $\nn$ and all rational numbers $(p_i)_{i\in F}$, there exists $n\in \nn$ such that $\sum_{i\in F}p_i d_i=d_n$.     We let $\textbf{SB}$ denote the subset of $F(C(2^\nn))$ consisting of those closed subsets of $C(2^\nn)$ which are linear subspaces.    This is easily seen to be a Borel subset of $F(C(2^\nn))$, whence it is also a standard Borel space.  From now on, we will treat $\textbf{SB}$ as a Polish space. However, as we are not concerned with the particular Polish topology on $\textbf{SB}$ which generates the Effros-Borel $\sigma$-algebra as its Borel $\sigma$-algebra, we will fix a Polish topology on $\textbf{SB}$ which generates the Effros-Borel $\sigma$-algebra and such that each selector $d_n$ is continuous with respect to this topology.   Define $r_n(X)=d_n(X)$ if $\|d_n(X)\|\leqslant 1$ and $r_n(X)=d_n(X)/\|d_n(X)\|$ if $\|d_n(X)\|>1$.  Note that each $r_n$ is also continuous and $\|r_n(X)\|\leqslant 1$ for each $n\in\nn$ and $X\in \textbf{SB}$.    We also let $\mathcal{L}$ denote the set of all triples $(X,Y, (y_n)_{n=1}^\infty)\in \textbf{SB}\times \textbf{SB}\times C(2^\nn)^\nn$ such that $y_n\in Y$ for all $n\in\nn$ and there exists $k\in \nn$ such that for all $n\in\nn$ and scalars $(a_i)_{i=1}^n$, $$\|\sum_{i=1}^n a_i y_i\|\leqslant k\|\sum_{i=1}^n a_id_i(X)\|.$$  It is easy to see that this is a Borel subset of $\textbf{SB}\times \textbf{SB}\times C(2^\nn)^\nn$, and is also therefore a standard Borel space.    We fix a Polish topology on $\mathcal{L}$ stronger than the topology inherited as a subspace of the product $\textbf{SB}\times\textbf{SB}\times C(2^\nn)^\nn$ and which generates the Effros-Borel $\sigma$-algebra. Note that the functions $(X, Y, (y_n)_{n=1}^\infty)\mapsto d_m(X), r_m(X)$ are still continuous for all $m\in\nn$.   This is the coding of all operators between separable Banach spaces introduced in \cite{BF}. That is, for any $(X, Y, (y_n)_{n=1}^\infty)\in \mathcal{L}$, the map $A_0:\{d_n(X):n\in\nn\}\to Y$ given by $A_0d_n(X)=y_n$ is well-defined and extends to a continuous, linear operator $A:X\to Y$. Conversely, if $A:X\to Y$ is a continuous, linear operator for $X,Y\in \textbf{SB}$, then $(X, Y, (Ad_n(X))_{n=1}^\infty)\in \mathcal{L}$.

We also recall the coding of dual spaces.  Let $H=[-1,1]^\nn$, endowed with a Polish topology such that the coordinate functional $(a_i)_{i=1}^\infty\mapsto a_n$ is continuous for each $n\in\nn$, and the map $(a_i)_{i=1}^\infty\mapsto \|(a_i)_{i=1}^\infty\|_\infty$ is continuous. We can see that such a topology exists by first endowing $H$ with its product topology $\tau$ and then using Proposition \ref{improv} to find a finer Polish topology $\tau'$ such that  $(a_i)_{i=1}^\infty\mapsto \|(a_i)_{i=1}^\infty\|_\infty$, which is Borel with respect to $\tau$, is continuous with respect to $\tau'$.    We leave this topology fixed throughout.     Given $X\in \textbf{SB}$ and $x^*\in B_{X^*}$, we define $$H\ni f_{x^*}= (x^*(r_1(X)), x^*(r_2(X)), x^*(r_3(X)), \ldots).$$   We let $K_X=\{f_{x^*}\in H: x^*\in B_{X^*}\}$.   We define $D\subset \textbf{SB}\times H$ by $(X,f)\in D\Leftrightarrow f\in K_X$.    Then $D$ is a Borel set and the bijective identification $B_{X^*}\ni x^*\leftrightarrow f_{x^*}\in K_X$ is isometric  (see properties P10-P12 from \cite[Page 12]{Dodos}).  More generally, for $x^*_1, \ldots, x^*_n\in K_X$ and scalars $(a_i)_{i=1}^n$, $\|\sum_{i=1}^n a_i x^*_i\|_{X^*}=\|\sum_{i=1}^n a_i f_{x^*_i}\|_\infty$.  

We also remark that since $2^\nn$ with its product topology is compact and the subset $[\nn]$ of infinite subsets of $\nn$ is $G_\delta$ in $2^\nn$, $[\nn]$ with its inherited topology is a Polish space.

Our first result regarding this is that functional evaluation is Borel.

\begin{lemma} The set $\mathcal{E}:=\{(Y,y, f_{y^*})\in \textbf{\emph{SB}}\times C(2^\nn)\times D: y\in Y, f_{y^*}\in K_Y\}$ is Borel and the map $(Y, y, f_{y^*})\mapsto y^*(y)$ is Borel from $\mathcal{E}$ to $\rr$.  
\label{hardest part}
\end{lemma}

\begin{proof} The map $(Y, y, f_{y^*})\mapsto (Y, f_{y^*})$ is continuous and therefore the set of $(Y, y, f_{y^*})$ such that $f_{y^*}\in Y$ is the set of $(Y, y, f_{y^*})$ such that $(Y, f_{y^*})\in D$ is Borel. It is known that the set of $(Y, y)\in\textbf{SB}\times C(2^\nn)$ such that $y\in Y$ is Borel (see \cite[Property (P4), Page 10]{Dodos}). Thus $\mathcal{E}$ is Borel.

To prove that evaluation is Borel, it is sufficient to prove that evaluation is continuous when $\textbf{SB}$ and $H$ are endowed with the topologies we have fixed above. We therefore proceed assuming that for each $n\in\nn$, $d_n$ and $r_n$ are continuous on $\textbf{SB}$ for each $n\in\nn$, and $(a_i)_{i=1}^\infty\mapsto a_n$, $(a_i)_{i=1}^\infty\mapsto \|(a_i)_{i=1}^\infty\|_\infty$ are continuous on $H$.  Assume $(Y, y, f_{y^*})\in \mathcal{E}$ is the limit of a sequence $((Y_n, y_n, f_{y^*_n}))_{n=1}^\infty\subset \mathcal{E}$.    

Define $T:[0, \infty)\to [1, \infty)$ by  $T(x)=1$ if $0\leqslant x\leqslant 1$ and $T(x)=x$ for all $x>1$.  Note that $T$ is $1$-Lipschitz. Fix $\ee>0$ and $j\in\nn$ such that $\|y-d_j(Y)\|<\ee$.      Fix $m\in\nn$ such that for all $n\geqslant m$, $\|y_n-y\|<\ee$, $|f_{y^*}(j)-f_{y^*_n}(j)|<\ee$, and $\|d_j(Y)-d_j(Y_n)\|<\ee$.  Now let us note that $d_j(Y)=T(\|d_j(Y)\|) r_j(Y)$, so $$y^*(d_j(Y))= T(\|d_j(Y)\|) f_{y^*}(j).$$   Similarly, for any $n\in\nn$, $$y^*_n(d_j(Y_n))=T(\|d_j(Y_n)\|)f_{y^*_n}(j).$$   By the triangle inequality, for any $n\geqslant m$, $$\|y_n-d_j(Y_n)\|\leqslant \|y_n-y\|+\|y-d_j(Y)\|+\|d_j(Y)-d_j(Y_n)\|<3\ee.$$   Thus $$    |y^*(y)-T(\|d_j(Y)\|)f_{y^*}(j)|=|y^*(y)-y^*(d_j(Y))|\leqslant \|y^*\|\|y-d_j(Y)\|\leqslant \|y-d_j(Y)\| <\ee $$   and for any $n\geqslant m$, $$    |y^*_n(y_n)-T(\|d_j(Y_n)\|)f_{y^*_n}(j)|=|y^*_n(y_n)-y^*_n(d_j(Y_n))|\leqslant \|y^*_n\|\|y_n-d_j(Y_n)\|\leqslant \|y_n-d_j(Y_n)\| <3\ee. $$  From this it follows that for any $n\geqslant m$, \begin{align*} |y^*(y)-y^*_n(y_n)| & \leqslant |y^*(y)- T(\|d_j(Y)\|)f_{y^*}(j)| + |T(\|d_j(Y)\|)f_{y^*}(j)- T(\|d_j(Y)\|)f_{y^*_n}(j)| \\ & + |T(\|d_j(Y)\|)f_{y^*_n}(j) - T(\|d_j(Y_n)\|)f_{y^*_n}(j)| + |T(\|d_j(Y_n)\|)f_{y^*_n}(j) - y^*_n(y_n)|  \\ & \leqslant \ee + T(\|d_j(Y)\|)\ee  + \ee + 3\ee \\ & \leqslant (5+\|y\|+\ee)\ee.  \end{align*}  Since $\ee>0$ was arbitrary, we are done.

\end{proof}

For the remainder of this work, when an ideal is denoted by a fraktur letter (with subscripts), the associated subsets of $\mathcal{L}$ and $\textbf{SB}$ will be denoted by the same letter (with the same subscripts) in calligraphic and bold fonts, respectively. That is, for an ideal $\mathfrak{I}$, we let $\mathcal{I}$ denote subset of $\mathcal{L}$ consisting of those members $(X, Y, (y_n)_{n=1}^\infty)$ of $\mathcal{L}$ such that the unique continuous extension of the function $d_n(X)\mapsto y_n$ lies in $\mathfrak{I}$.   We let $\textbf{I}$ denote the subset of $\textbf{SB}$ consisting of those $X\in \textbf{SB}$ such that $X\in \textsf{I}$ (equivalently, such that $I_X\in \mathfrak{I}$).

\begin{rem}\upshape The map $\Phi:\textbf{SB}\to \mathcal{L}$ given by $\Phi(X)=(X,X, (d_n(X))_{n=1}^\infty)$ is Borel. From this it follows that for any ideal $\mathfrak{I}$ of operators, if $\mathcal{I}$ is $\Pi_1^1$ (resp. $\Pi_2^1$), then $\textbf{I}$ is $\Pi_1^1$ (resp. $\Pi_2^1$).    Therefore to provide an upper estimate on the complexities of $\mathcal{I}$ and $\textbf{I}$, it is sufficient to provide that upper estimate only for $\mathcal{I}$.

Similarly, in order to show that $\mathcal{I}$ and $\textbf{I}$ are $\Pi_1^1$-hard (resp. $\Pi_2^1$-hard), it is sufficient to show that $\textbf{I}$ is $\Pi_1^1$-hard (resp. $\Pi_2^1$-hard).  To see this,  if $P$ is a Polish space,  $C\subset P$ is a $\Pi_1^1$ subset of $P$,  and $\Psi:P\to \textbf{I}$ is a Borel map such that $\Psi^{-1}(\textbf{I})=C$, then $\Phi\circ\Psi:P\to \mathcal{L}$ is a Borel reduction of $\mathcal{I}$ to $C$. A similar statement holds for $\Pi_2^1$-hard sets.   

\label{oopyal}
\end{rem}

Given $\xi<\omega_1$ and $M\in [\nn]$, there exists a unique, non-empty, finite initial segment of $M$ which is a maximal member of $\mathcal{S}_\xi$.  We denote this initial segment by $M|_\xi$.   Given a Banach space $X$, a  sequence $(x_n)_{n=1}^\infty\subset X$, $\xi<\omega_1$, and $M\in[\nn]$, we let $$\Xi_\xi((x_n)_{n=1}^\infty,  M)= \min\{\|x\|: x\in \text{co}(x_n: n\in M|_\xi)\}.$$

For a Banach space $X$, $(x_n)_{n=1}^\infty\subset X$, and $M=(m_n)_{n=1}^\infty\in [\nn]$, let us say that the pair $((x_n)_{n=1}^\infty, M)$ has $D_\xi$ provided that for any $k\in \nn$ and $N\in [M]$ with $\min N\geqslant m_k$, $\Xi_\xi((x_n)_{n=1}^\infty, N)\leqslant 1/k$.

\begin{proposition} Let $(x_n)_{n=1}^\infty$ be a sequence in the Banach space $X$. Let $\xi<\omega_1$.  \begin{enumerate}[(i)]  \item $(x_n)_{n=1}^\infty$ fails to be $\xi$-weakly null if and only if there exist $m\in\nn$ and $M\in [\nn]$ such that for all $N\in [M]$, $\Xi_\xi((x_n)_{n=1}^\infty, N)\geqslant 1/m$. \item If $(x_n)_{n=1}^\infty\subset X$ is $\xi$-weakly null, then for any $M_0\in [\nn]$, there exists $M\in [M_0]$ such that $((x_n)_{n=1}^\infty, M)$ has $D_\xi$.  \item If $M=(m_n)_{n=1}^\infty\in [\nn]$ is such that $((x_n)_{n=1}^\infty, M)$ has $D_\xi$, then $(x_{m_n})_{n=1}^\infty$ is $\xi$-weakly null.     \end{enumerate}
\label{fras}
\end{proposition}

\begin{proof} First, for each $m\in\nn$, let $\mathcal{V}_m$ denote the set of subsets $M\in [\nn]$ such that $$\Xi_\xi((x_n)_{n=1}^\infty, M)\geqslant 1/m.$$  Since this is a closed set, the infinite Ramsey theorem yields that for each $m\in\nn$ and $M\in[\nn]$, there exists $N\in [M]$ such that either $[N]\subset \mathcal{V}_m$ or $[N]\cap \mathcal{V}_m=\varnothing$.    From this and a standard diagonalization, we establish the dichotomy that either there exist $m\in\nn$ and $M\in[\nn]$ such that for all $N\in[M]$, $\Xi_\xi((x_n)_{n=1}^\infty, N)\geqslant 1/m$, or for every $M_0\in [\nn]$, there exists $M\in[\nn]$ such that $((x_n)_{n=1}^\infty, M)$ has $D_\xi$.     We will show that the first of these two conditions is equivalent to $(x_n)_{n=1}^\infty$ failing to be $\xi$-weakly null, which will yield both $(i)$ and $(ii)$.

First suppose that there exist $m\in \nn$ and $M=(m_i)_{i=1}^\infty\in[\nn]$ such that  for all $N\in [M]$, $\Xi_\xi((x_n)_{n=1}^\infty, N)\geqslant 1/m$.    Now fix $F\in \mathcal{S}_\xi$ and let $E=(m_n)_{n\in F}\in \mathcal{S}_\xi$. Let $N$ be any infinite subset of $M$ such that $E$ is an initial segment of $N$. Now let $(a_n)_{n\in F}$ be non-negative numbers summing to $1$ and note that $$\|\sum_{n\in F}a_nx_{m_n}\|\geqslant \Xi_\xi((x_n)_{n=1}^\infty, N)\geqslant 1/m.$$    This yields that $(x_{m_n})_{n=1}^\infty$ is an $\ell_1^\xi+$-spreading model, and $(x_n)_{n=1}^\infty$ is not $\xi$-weakly null.

Now suppose that $(x_n)_{n=1}^\infty$ is not $\xi$-weakly null.    If $(x_n)_{n=1}^\infty$ is not weakly null, then there exist $m\in\nn$ and $M=(m_n)_{n=1}^\infty$ such that $\inf \{\|x\|: x\in \text{co}(x_{m_n}:n\in\nn)\}\geqslant 1/m$.   Then $\Xi_\xi((x_n)_{n=1}^\infty, N)\geqslant 1/m$ for all $N\in [M]$.  Now suppose that $(x_n)_{n=1}^\infty$ is weakly null but not $\xi$-weakly null.  This means there exist $\ee>0$ and $r_1<r_2<\ldots$ such that $(x_{r_n})_{n=1}^\infty$ is $2$-basic and for each $F\in \mathcal{S}_\xi$ and $x\in \text{co}(x_{r_n}:n\in F)$, $\|x\|\geqslant \ee$.  Now choose $1=s_1<s_2<\ldots$ such that for each $n\in\nn$, $s_{n+1}>r_{s_n}$.   Let $m_n=r_{s_n}$ and $M=(m_n)_{n=1}^\infty$.    Fix $N\in [M]$ and let $N|_\xi=(r_{s_{t_i}})_{i=1}^l$. Then $F:=(s_{t_i})_{i=2}^l$ is a spread of $(r_{s_{t_i}})_{i=1}^{l-1}$ and therefore lies in $\mathcal{S}_\xi$.      Fix non-negative scalars $(a_i)_{i=1}^l$ summing to $1$ and note that \begin{align*} \|\sum_{i=1}^l a_i x_{r_{s_{t_i}}}\| & \geqslant \frac{1}{3}\max\Bigl\{a_1\|x_{r_{s_{t_1}}} \|, \|\sum_{i=2}^l a_i x_{r_{s_{t_i}}}\|\Bigr\} \geqslant \frac{\ee}{3}\max\Bigl\{a_1, \sum_{i=2}^l a_i\Bigr\} \geqslant \ee/6.  \end{align*}  Thus for any $N\in [M]$, $\Xi_\xi((x_n)_{n=1}^\infty, M)\geqslant \ee/6$.   Fixing $m>6/\ee$, we conclude the stated equivalence. This yields $(i)$ and $(ii)$.

$(iii)$ If $((x_n)_{n=1}^\infty, M)$ has $D_\xi$ and $(x_{m_n})_{n=1}^\infty$ is not $\xi$-weakly null, we may argue as in the previous paragraph to find $m\in\nn$,  $r_1<r_2<\ldots$, and $s_1<s_2<\ldots$ such that for each $N\in[(r_{s_n})_{n=1}^\infty]$, $\Xi_\xi((x_n)_{n=1}^\infty, N)\geqslant 1/m$, with the added condition that $(r_n)_{n=1}^\infty \in [(m_n)_{n=1}^\infty]$.  Now if we fix $k>m$ and $N\in [M]$ with $\min N\geqslant k$, these two conditions yield that $$1/m\leqslant \Xi_\xi((x_n)_{n=1}^\infty, N)\leqslant 1/k,$$ which is a contradiction.

\end{proof}

\begin{lemma} The sets $$W=\{(X, (n_i)_{i=1}^\infty)\in \textbf{\emph{SB}}\times \nn^\nn: (d_{n_i}(X))_{i=1}^\infty\text{\ is weakly null}\}$$ and $$W^*=\{(Y, (f_i)_{i=1}^\infty)\in \textbf{\emph{SB}}\times H^\nn: (f_i)_{i=1}^\infty\subset K_Y, (f_i)_{i=1}^\infty\text{\ is weakly null in\ } \ell_\infty\}$$ are $\Pi_1^1$.  For $\xi<\omega_1$, the sets $$W_\xi=\{(X, (n_i)_{i=1}^\infty,M)\in \textbf{\emph{SB}}\times\nn^\nn\times [\nn]:  ((d_{n_i}(X))_{i=1}^\infty, M) \text{\ has\ }D_\xi\}$$ and $$W_\xi^*=\{(Y, (f_i)_{i=1}^\infty, P)\in \textbf{\emph{SB}}\times H^\nn\times [\nn]: (f_i)_{i=1}^\infty\subset K_Y, ((f_i)_{i=1}^\infty, P)\text{\ has\ }D_\xi\text{\ in\ }\ell_\infty\}$$    are Borel.

\label{boring}
\end{lemma}

\begin{proof}   First let $C$ denote the set of those $(X, (n_i)_{i=1}^\infty, M, p)\in \textbf{SB}\times\nn^\nn\times [\nn]\times \nn$ such that for all $k\in\nn$ and non-negative scalar sequences $(a_i)_{i=1}^k$, $\|\sum_{i=1}^k a_i d_{n_{m_i}}(X)\|\geqslant 1/p$. Here, $M=(m_i)_{i=1}^\infty$.     It is evident that $C$ is closed. Let $\pi:\textbf{SB}\times \nn^\nn\times [\nn]\times\nn\to \textbf{SB}\times \nn^\nn$ be the projection and note that, by the Mazur lemma, $(\textbf{SB}\times\nn^\nn)\setminus \pi(C)$ is the set $W$.

Now let $A$ denote the set of those $(Y, (f_i)_{i=1}^\infty, M, p)\in \textbf{SB}\times H^\nn\times [\nn]\times \nn$ such that for all $i\in\nn$, $f_i\in K_Y$.   Let $B$ denote the set of those $(Y, (f_i)_{i=1}^\infty, M, p)\in \textbf{SB}\times H^\nn\times [\nn]\times \nn$ such that for all $k\in\nn$ and non-negative scalars $(a_i)_{i=1}^k$, $\|\sum_{i=1}^k a_i f_{m_i}\|_\infty\geqslant 1/p$.    Since $\{(Y, f)\in \textbf{SB}\times H: f\in K_Y\}$ is Borel, $A$ is Borel.    It is obvious that the set $B$ is closed, as we have assumed a topology on $H$ making the supremum norm continuous.  Then $A\cap B$ is Borel. Let $\pi:\textbf{SB}\times H^\nn\times [\nn]\times \nn\to \textbf{SB}\times H^\nn$ be the projection and note that, by another appeal to the Mazur lemma, $W^*=(\textbf{SB}\times H^\nn)\setminus \pi(A\cap B)$.

We next show that, with our fixed topologies, the set $W_\xi$ is closed.   To that end, fix $(X, (n_i)_{i=1}^\infty, M)\in  (\textbf{SB}\times \nn^\nn\times [\nn])\setminus W_\xi$.   This means there exist $\ee>0$, $k\in\nn$, and $L\in [M]$ such that $\min L\geqslant m_k$ and for all non-negative scalars $(a_i)_{i\in L|_\xi}$ summing to $1$, $$\|\sum_{i\in L|_\xi}a_i d_{n_i}(X)\|>\ee+1/k.$$ Here, $L=(l_i)_{i=1}^\infty$ and $M=(m_i)_{i=1}^\infty$.  Let $t=\max L|_\xi$.       Let $U_1$ denote the set of $Y\in \textbf{SB}$ such that for each $1\leqslant i\leqslant n_t$, $\|d_i(Y)-d_i(X)\|<\ee$.     By continuity of the selectors, $U_1$ is open in $\textbf{SB}$.    Let $U_2$ denote the subset of $\nn^\nn$ consisting of those $(p_i)_{i=1}^\infty$ such that $p_i=n_i$ for all $1\leqslant i\leqslant t$, which is open. Let $U_3$ denote the subset of $[\nn]$ consisting of those $Q\in[\nn]$ such that for all $1\leqslant i\leqslant \max\{m_k, t\}$, $1_M(i)=1_Q(i)$.  This is an open set in $[\nn]$.  Now let $U=U_1\times U_2\times U_3$, which is an open subset of $\textbf{SB}\times \nn^\nn\times [\nn]$ containing $(X, (n_i)_{i=1}^\infty, M)$.     We claim that $U\cap W_\xi=\varnothing$. Indeed, suppose $(Y, (p_i)_{i=1}^\infty, Q)\in U$.     Let us note that, if $L|_\xi=(l_1, \ldots, l_s)$, $l_s=t$. Since $Y\in U_1$, it follows that $\|d_i(Y)-d_i(X)\|<\ee$ for all $1\leqslant i\leqslant n_{l_s}$.    Since $(p_i)_{i=1}^\infty\in U_2$,   $n_i=p_i$ for all $i\leqslant l_s$.  Since $Q\in U_3$, if $j\in\nn$ is such that $t=m_j$, $q_i=m_i$ for all $i\leqslant j$. The last condition implies that $(l_1, \ldots, l_s)$ is an initial segment of some infinite subset $R$ of $Q$ such that $\min R=\min L\geqslant m_k=q_k$.     Moreover, since $(l_1, \ldots, l_s)$ is a maximal member of $\mathcal{S}_\xi$ which is also an initial segment of $R$, $R|_\xi=L|_\xi$.    Then for any non-negative scalars $(a_i)_{i\in R|_\xi}=(a_i)_{i\in L|_\xi}$ summing to $1$, $$\|\sum_{i\in R|_\xi}a_i d_{p_i}(Y)\| = \|\sum_{i\in L|_\xi}a_i d_{n_i}(Y)\|\geqslant \|\sum_{i\in L|_\xi}a_i d_{n_i}(X)\|-\sum_{i\in L|_\xi}a_i \|d_{n_i}(X)-d_{n_i}(Y)\|>\ee+1/k-\ee=1/k.$$  This yields that $(Y, (p_i)_{i=1}^\infty, Q)\notin W_\xi$, and  $W_\xi$ is closed.

Now let $A$ denote the set of $(Y, (f_i)_{i=1}^\infty, M)\in \textbf{SB}\times H^\nn\times [\nn]$ such that $f_i\in K_Y$ for all $i\in\nn$ and let $B$ denote the set of $(Y, (f_i)_{i=1}^\infty, M)\in \textbf{SB}\times H^\nn\times [\nn]$ such that $((f_i)_{i=1}^\infty, M)$ has $D_\xi$ in $\ell_\infty$.  If $(Y, (f_i)_{i=1}^\infty, M)\notin B$, there exist $\ee>0$, $k\in\nn$, $L\in [M]$ such that $\min L\geqslant m_k$ and for all non-negative scalars $(a_i)_{i\in L|_\xi}$ summing to $1$, $$\|\sum_{i\in L_\xi}a_i f_i\|_\infty > \ee+1/k.$$ Let $t=\max L|_\xi$.  We let $U_1$ be the set of those $(g_i)_{i=1}^\infty\in H^\nn$ such that $\|f_i-g_i\|_\infty<\ee$ for all $1\leqslant i\leqslant t$. Let $U_2$ be the set of those $Q\in [\nn]$ such that $1_M(i)=1_Q(i)$ for all $i\leqslant \max\{ m_k,t\}$. Then as in the previous paragarph, if $(Z, (g_i)_{i=1}^\infty, Q)\in \textbf{SB}\times U_1\times U_2$, there exists $R\in [Q]$ with $\min L=\min R\geqslant m_k=q_k$ such that $R|_\xi=L|_\xi$ and $$\|\sum_{i\in R|_\xi}a_i g_i\|_\infty =\|\sum_{i\in L|_\xi} a_ig_i\|_\infty \geqslant \|\sum_{i\in L|_\xi}a_if_i\|_\infty-\sum_{i\in L|_\xi}a_i\|f_i-g_i\|_\infty>1/k$$ for all non-negative scalars $(a_i)_{i\in R|_\xi}$ summing to $1$.  This yields that $(Z, (g_i)_{i=1}^\infty, Q)\notin B$. This yields that $B$ is closed. Since $A\cap B=W_\xi^*$ and $A$ is Borel, $W_\xi^*$ is Borel.

\end{proof}

We are now ready to prove the upper estimates.

\begin{proposition} \begin{enumerate}[(i)]\item For $0\leqslant \zeta<\xi\leqslant \omega_1$, $\mathcal{G}_{\xi, \zeta}$ and $\textbf{\emph{G}}_{\xi, \zeta}$ are $\Pi_2^1$, and $\Pi_1^1$ if $\xi<\omega_1$. \item For $1\leqslant \zeta,\xi\leqslant \omega_1$, $\mathcal{M}_{\xi, \zeta}$ and $\textbf{\emph{M}}_{\xi, \zeta}$ are $\Pi_2^1$, and $\Pi_1^1$ if $\zeta, \xi<\omega_1$.\end{enumerate}

\label{upper}
\end{proposition}

\begin{proof}$(i)$ It suffices to prove that $\mathcal{G}_{\xi, \zeta}$ is $\Pi_1^1$ if $0\leqslant \zeta<\xi<\omega_1$ and that $\mathcal{G}_{\omega_1, \zeta}$ is $\Pi_2^1$ for any $\zeta<\omega_1$.  The desired membership of the classes of spaces then follows from Remark \ref{oopyal}. 

First fix $\xi<\omega_1$. Let $B$ denote the set of $((X, Y, (y_i)_{i=1}^\infty), (n_i)_{i=1}^\infty, M, p)\in \mathcal{L}\times \nn^\nn\times [\nn]\times \nn$ such that $((d_{n_i}(X))_{i=1}^\infty, M)$ has $D_\xi$ and let $C$ denote the set of $((X, Y, (y_i)_{i=1}^\infty), (n_i)_{i=1}^\infty, M, p)\in \mathcal{L}\times \nn^\nn\times [\nn]\times \nn$ such that for all $F\in \mathcal{S}_\zeta$ and non-negative scalars $(a_i)_{i\in F}$ summing to $1$, $\|\sum_{i\in F}a_i y_{n_i}\|\geqslant 1/p$. It is obvious that $C$ is Borel (and actually closed with our fixed topologies), and we know that $B$ is Borel by Lemma \ref{boring}.  Let $\pi:\mathcal{L}\times \nn^\nn\times [\nn]\times \nn\to \mathcal{L}$ be the projection and note that $\pi(B\cap C)$ is $\Sigma_1^1$. In order to show that $\mathcal{G}_{\xi, \zeta}$ is $\Pi_1^1$, it suffices to show that $\mathcal{L}\setminus \pi(B\cap C)=\mathcal{G}_{\xi, \zeta}$.    

If $(X,Y, (y_i)_{i=1}^\infty)\in \mathcal{L}\setminus \mathcal{G}_{\xi, \zeta}$, then there exists a $\xi$-weakly null sequence $(x_i)_{i=1}^\infty\subset X$ whose image under the operator associated with the triple $(X, Y, (y_i)_{i=1}^\infty)$ is an $\ell_1^\zeta+$-spreading model.   By perturbing, we may assume $x_i=d_{n_i}(X)$ for some $(n_i)_{i=1}^\infty\in \nn^\nn$. By Proposition \ref{fras}, there exists $M\in[\nn]$ such that $((d_{n_i}(X))_{i=1}^\infty, M)$ has $D_\xi$.  Furthermore, since the image of $d_{n_i}(X)$ under the operator associated with the triple is $y_{n_i}$,  $(y_{n_i})_{i=1}^\infty$ is an $\ell_1^\zeta+$-spreading model.     This yields the existence of some $p\in\nn$ such that $((X, Y, (y_i)_{i=1}^\infty), (n_i)_{i=1}^\infty, M, p)\in B\cap C$.  Therefore $\mathcal{L}\setminus \pi(B\cap C)\subset \mathcal{G}_{\xi, \zeta}$.    For the reverse inclusion, assume that $(X, Y, (y_i)_{i=1}^\infty)\in \pi(B\cap C)$.  Fix $(n_i)_{i=1}^\infty \in \nn^\nn$, $M\in [\nn]$, $p\in\nn$ such that $((X, Y, (y_i)_{i=1}^\infty), (n_i)_{i=1}^\infty, M, p)\in B\cap C$.  Then by Proposition \ref{oopyal}, $(d_{n_{m_i}}(X))_{i=1}^\infty$ is $\xi$-weakly null, while $(y_{n_{m_i}})_{i=1}^\infty$ is an $\ell_1^\zeta+$-spreading model. Thus $(X, Y, (y_i)_{i=1}^\infty)\in \mathcal{L}\setminus \mathcal{G}_{\xi, \zeta}$.

The $\xi=\omega_1$ case is similar. We replace $W_\xi$ with $W$ from Lemma \ref{boring}.

$(ii)$ Suppose that $\xi, \zeta<\omega_1$.  Let $P=\mathcal{L}\times\nn^\nn\times H^\nn\times [\nn]\times \nn$.   We let $A,B,C$ be the subsets of $P$ consisting of those $((X, Y, (y_i)_{i=1}^\infty), (n_i)_{i=1}^\infty, (f_i)_{i=1}^\infty, M, p)\in P$ such that \begin{equation} \text{for all\ }i\in\nn, f_i\in K_Y, \tag{A}\end{equation} \begin{equation}((d_{n_i}(X))_{i=1}^\infty, M)\text{\ has\ }D_\xi,\tag{B}\end{equation} \begin{equation} ((f_i)_{i=1}^\infty, M)\text{\ has\ }D_\zeta.\tag{C}\end{equation}  Note that $A,B,C$ are Borel. Let $E$ denote the subset of $A$ consisting of those $((X, Y, (y_i)_{i=1}^\infty), (n_i)_{i=1}^\infty, M, p)$ such that for all $i\in\nn$, $f_i(y_{n_i})\geqslant 1/p$.    Note that $E$ is a Borel subset of $A$, and is therefore Borel in $P$.  Let $\pi:P\to \mathcal{L}$ be the projection. Arguing as in the first paragraph, we deduce that $\mathcal{L}\setminus \pi(E\cap B\cap C)=\mathcal{M}_{\xi, \zeta}$.   Thus this set is $\Pi_1^1$.  Here we are using the fact that for $f_{y_1^*}, \ldots, f_{y_n^*}\in K_Y$ and scalars $(a_i)_{i=1}^n$, $\|\sum_{i=1}^n a_i f_{y^*_i}\|_\infty=\|\sum_{i=1}^n a_i y^*_i\|_{Y^*}$.   Therefore if $f_i\in K_Y$, $f_i=f_{y^*_i}$, and $((f_i)_{i=1}^\infty, M)$ has $D_\zeta$ in $\ell_\infty$, $((y^*_i)_{i=1}^\infty, M)$ has $D_\zeta$ in $Y^*$.    Therefore $(y^*_{m_i})_{i=1}^\infty$ is $\zeta$-weakly null.

 Now if $\xi=\omega_1$ and $\zeta<\omega_1$, we replace $B$ above with the set of $((X, Y, (y_i)_{i=1}^\infty), (n_i)_{i=1}^\infty, M, p)$ such that $(d_{n_i}(X))_{i=1}^\infty$ is weakly null, which is $\Pi_1^1$. This gives that the resulting set $\mathcal{M}_{\omega_1, \zeta}$ is $\Pi_2^1$.   If $\xi<\omega_1$ and $\zeta=\omega_1$, we replace the set $C$ above with the set of $((X, Y, (y_i)_{i=1}^\infty), (n_i)_{i=1}^\infty, M, p)$ such that $(f_i)_{i=1}^\infty$ is weakly null.    If $\xi=\zeta=\omega_1$, we make both of these replacements of $B$ and $C$.

\end{proof}

Throughout, for a given set $S$, $2^S$ will be topologized with the product topology. Given a set $\Lambda$, we let $\textbf{Tr}(\Lambda)$ denote the subset of $2^{\Lambda^{<\nn}}$ consisting of those subsets which contain all initial segments of their members.     Let $\textbf{Tr}=\textbf{Tr}(\nn)$.      Let $\textbf{WF}$ and $\textbf{IF}$, respectively, denote the subsets of $\textbf{Tr}$ consisting of well-founded and ill-founded trees.  Let us recall that $T$ is \emph{ill-founded} if there exists $(n_i)_{i=1}^\infty\in \nn^\nn$ such that $(n_i)_{i=1}^l\in T$ for all $l\in\nn$, and $T$ is \emph{well-founded} otherwise.    Let us also recall that $\textbf{Tr}$ with the topology inherited from $2^{\Lambda^{<\nn}}$ is  a Polish space and $\textbf{WF}$ is a $\Pi_1^1$-complete subset of $\textbf{Tr}$ \cite[Theorem A$.9$, page $130$]{Dodos}. Let us note that a regular family, by compactness, is always identified with a well-founded tree on $\nn$.

For $T\in \text{Tr}(2\times \nn)$ and $\sigma=(\ee_n)_{n=1}^\infty\in 2^\nn$, let $T(\sigma)=\varnothing$ if $T=\varnothing$ and otherwise let $$T(\sigma)=\{\varnothing\}\cup \{(n_i)_{i=1}^l\in \nn^{<\nn}: (\ee_i, n_i)_{i=1}^l\in T\}.$$  

Let us define the  subset $C$ of $\text{Tr}(2\times \nn)$  by $$C=\{T\in \text{Tr}(2\times \nn): (\forall \sigma\in 2^\nn)(T(\sigma)\in \textbf{IF})\}.$$    Then $C$ is $\Pi_2^1$-complete \cite[Lemma $4.1$]{K}.

For a finite sequence $v=(n_1, \ldots, n_k)$ of natural numbers, let $\overline{v}=(n_1, n_1+n_2, \ldots, n_1+\ldots +n_k)$. For an infinite sequence $v=(n_1, n_2, \ldots)$, let $\overline{v}=(n_1, n_1+n_2, \ldots)$.     Let $\overline{\varnothing}=\varnothing$.  For the following proposition, let us recall that we identify subsets of $\nn$ with sequences in $\nn$ in the natural way. A subset is identified with the sequence obtained by listing the members of that subset in strictly increasing order. Therefore a regular family is identified with a tree on $\nn$.

\begin{proposition} Suppose $\mathcal{F}$ is a regular family.  Define the map $U_\mathcal{F}:\text{\emph{Tr}}(2\times \nn)\to \text{\emph{Tr}}(2\times \nn)$ by letting $\varnothing\in U_\mathcal{F}(T)$ and by letting  $(\ee_i, n_i)_{i=1}^k\in U_\mathcal{F}(T)$ if and only if either $(\ee_i, n_i)_{i=1}^k\in T$ or $\overline{(n_i)_{i=1}^k}\in \mathcal{F}$. Then $T\mapsto U_\mathcal{F}(T)$ is continuous. Furthermore, $U_\mathcal{F}(T)\in C$ if and only if $T\in C$. 

\label{crusoe}

\end{proposition}

\begin{proof} Fix $t\in (2\times \nn)^{<\nn}$.  Fix a sequence $T_n$ of trees on $2\times \nn$ converging to the tree $T$.  If $t=\varnothing$, then for all $n\in\nn$, $$1_{U_\mathcal{F}(T_n)}(t)=1=1_{U_\mathcal{F}(T)}(t).$$   If $t\neq \varnothing$, write $t=(\ee_i, n_i)_{i=1}^k$ and let $v=(n_i)_{i=1}^k$.  If $\overline{v}\in \mathcal{F}$, then for all $n\in\nn$,  $$1_{U_\mathcal{F}(T_n)}(t)=1=1_{U_\mathcal{F}(T)}(t).$$  Otherwise $$1_{U_\mathcal{F}(T_n)}(t)= 1_{T_n}(t)\to 1_T(t)=1_{U_\mathcal{F}(T)}(t).$$

Now let $I$ denote the set of all trees $T$ on $\nn$ such that if $(n_i)_{i=1}^k\in T$, then $n_1<\ldots <n_k$. Note that $\mathcal{F}\in I$.    Define the map $\Psi:\text{Tr}\to I$ by $\Psi(T)=\{\overline{v}: v\in T\}$ and note that $\Psi$ is a bijection. Note also that $\Psi(T)$ is well-founded if and only if $T$ is.  Furthermore, it is well-known that if $S,T$ are two trees, then $S\cup T$ is well-founded if and only if $S,T$ are.    From this it follows that for a tree $T$ on $\nn$, then $T$, $\Psi(T)$, $\Psi(T)\cup \mathcal{F}$, and $\Psi^{-1}(\Psi(T)\cup \mathcal{F})$ are all well-founded, or all ill-founded.

One can describe $U_\mathcal{F}$ by noting that for each $\sigma\in 2^\nn$ and $V\in \textbf{Tr}(2\times\nn)$,  $U_\mathcal{F}(V)(\sigma)= \Psi^{-1}(\Psi(V(\sigma))\cup \mathcal{F})$.      Now suppose that $V$ is a tree on $2\times \nn$.  Suppose that $V\in C$. Then for each $\sigma\in 2^\nn$, by the last paragraph applied with $T=V(\sigma)$,  $U_\mathcal{F}(V)(\sigma)=\Psi^{-1}(\Psi(V(\sigma))\cup \mathcal{F})$ is ill-founded. Since this holds for any $\sigma$, $U_\mathcal{F}(V)\in C$.   Now if $V\in \textbf{Tr}(2\times \nn)\setminus C$, then there exists $\sigma\in 2^\nn$ such that $V(\sigma)$ and $\Psi^{-1}(\Psi(V(\sigma))\cup \mathcal{F})$ are well-founded.    In this case, $U_\mathcal{F}(V)\in \textbf{Tr}(2\times \nn)\setminus C$.

\end{proof}

Let us recall that for a Banach space $R$ and an ordinal $0<\alpha<\omega_1$, a basis $(e_i)_{i=1}^\infty$ for $R$ is said to be \emph{asymptotic} $c_0^\alpha$ (resp. \emph{asymptotic} $\ell_1^\alpha$) in $R$ provided that there exists $a>0$ such that whenever $(x_i)_{i=1}^l$ is a block sequence with respect to $(e_i)_{i=1}^\infty$ such that $(\min \text{supp}(x_i))_{i=1}^l\in \mathcal{S}_\alpha$, $$\|\sum_{i=1}^l x_i\|\leqslant a\max_{1\leqslant i\leqslant l}\|x_i\|$$ $$(\text{resp.\ }\|\sum_{i=1}^l x_i\|\geqslant a\sum_{i=1}^l \|x_i\|).$$  Note that every seminormalized block sequence in a space with a basis which is asymptotic-$\ell_1^\alpha$ in the space is an $\ell_1^\alpha$-spreading model, and is therefore not $\alpha$-weakly null.

\begin{corollary} For each $0<\alpha<\omega_1$, there exist Borel maps $\mathfrak{S}_\alpha, \mathfrak{S}_\alpha^*:\text{\emph{Tr}}(2\times \nn)\to \textbf{\emph{SB}}$ such that \begin{enumerate}[(i)]\item if $T\in C$, then $\mathfrak{S}_\alpha(T)^*, \mathfrak{S}^*_\alpha(T)$ have the Schur property, and therefore $\mathfrak{S}_\alpha(T), \mathfrak{S}_\alpha^*(T)$ have the Dunford-Pettis property, and  \item if $T\notin C$, then $\mathfrak{S}_\alpha(T)$ (resp. $\mathfrak{S}_\alpha^*(T)$) has a complemented, reflexive subspace $R$ with a basis which is asymptotic $c_0^\alpha$ (resp. asymptotic $\ell_1^\alpha$) in $R$.  \end{enumerate}

\label{rob}

\end{corollary}

\begin{proof} For a tree $T$ on $\nn$, let us recall that $$[T]=\{(n_i)_{i=1}^\infty \in \nn^\nn: (\forall k\in\nn)((n_i)_{i=1}^k\in T)\}.$$ define $$\mathcal{M}_T=\{\varnothing\}\cup \{\{n_1, n_1+n_2, \ldots, n_1+\ldots +n_k\}: (n_i)_{i=1}^k\in T\}\cup \{\{n_1, n_1+n_2, \ldots\}: (n_i)_{i=1}^\infty\in [T]\}\in 2^\nn.$$   Note that $\mathcal{M}_T$ is compact.  Given $\sigma=(\ee_i)_{i=1}^\infty\in 2^\nn$ and $l\in\nn$, we let $\sigma|_l=(\ee_i)_{i=1}^l$.     Now for $T\in \textbf{Tr}(2\times \nn)$, we define the space $E_T$ to be the completion of $c_{00}(2^{<\nn}\setminus \{\varnothing\})$ with respect to the norm $$[\sum_{t\in 2^{<\nn}\setminus \{\varnothing\}}a_te_t]= \sup_{\sigma \in 2^\nn}  \|\sum_{l=1}^\infty a_{\sigma|_l} e_l\|_{\mathcal{M}_{T(\sigma)}}.$$ Here, for a compact set $\mathcal{M}\subset 2^\nn$, $\|\cdot\|_\mathcal{M}$ denotes the Tsirelson space $\mathcal{T}^*[\mathcal{M}, 1/2]$ as defined in \cite{K}.   Kurka showed that there exist Borel maps $\mathfrak{S}, \mathfrak{S}^*:\text{Tr}(2\times \nn)\to \textbf{SB}$ such that for each $T\in \text{Tr}(2\times \nn)$, $\mathfrak{S}(T)$ is isometric to $E_T$ and $\mathfrak{S}^*(T)$ is isometric to $E_T^*$. Kurka also showed that if $T\in C$, $E_T^*$ has the Schur property.     Moreover, it is easy to see that for any $\sigma=(\ee_i)_{i=1}^\infty\in 2^\nn$, $E_T$ contains a complemented copy of the space $\mathcal{T}^*[\mathcal{M}_{T(\sigma)},1/2]$, namely the closed span of the branch $(e_{\sigma|_l})_{l=1}^\infty$.

Now for $0<\alpha<\omega_1$,  let us define $\mathfrak{S}_\alpha =\mathfrak{S}\circ U_{\mathcal{S}_\alpha}$ and $\mathfrak{S}^*_\alpha= \mathfrak{S}^*\circ U_{\mathcal{S}_\alpha}$, where $U_{\mathcal{S}_\alpha}$ is as defined in Proposition \ref{crusoe}.     Thus these maps are Borel. Furthermore, if $T\in C$, then so is $U_{\mathcal{S}_\alpha}(T)$. By Kurka's result, for $T\in C$,  $\mathfrak{S}_\alpha(T)$ is isometric to $E_{U_{\mathcal{S}_\alpha}(T)}$, the dual of which has the Schur property.  Similarly, for $T\in C$,  $\mathfrak{S}_\alpha(T)^*$ is isometric to $E_{U_{\mathcal{S}_\alpha}(T)}^*$, which has the Schur property.   Now if $T\in \text{Tr}(2\times \nn)\setminus C$, there exists $\sigma\in 2^\nn$ such that $T(\sigma)$, and $U_{\mathcal{S}_\alpha}(T)(\sigma)$, are well-founded.   Let $\mathcal{M}=\mathcal{M}_{U_{\mathcal{S}_\alpha}(T)(\sigma)}$. Then since $\mathcal{M}$ contains only finite sets (that is, since $U_{\mathcal{S}_\alpha}(T)(\sigma)$ is well-founded), $[U_{\mathcal{S}_\alpha}(T)(\sigma)]=\varnothing$, it is well-known that $\mathcal{T}^*[\mathcal{M},1/2]$ is reflexive. By construction, $\mathcal{S}_\alpha\subset \mathcal{M}$, whence the basis of $\mathcal{T}^*[\mathcal{M},1/2]$ is asymptotic $c_0^\alpha$ and the basis of the dual space is asymptotic $\ell_1^\alpha$.  Then $\mathfrak{S}_\alpha(T)$ contains a complemented copy of the reflexive space $\mathcal{T}^*[\mathcal{M},1/2]$ with asymptotic $c_0^\alpha$ basis.  Since $\mathfrak{S}_\alpha^*(T)$ is isometric to $E^*_T$, it contains a complemented copy of the reflexive space $\mathcal{T}[\mathcal{M},1/2]$ with asymptotic $\ell_1^\alpha$ basis.

\end{proof}

\begin{rem}\upshape If $X$ is a Banach space with a complemented, reflexive subspace $R$ having a seminormalized basis which is asymptotic $c_0$ in $R$, then $X$ lies in $\complement\textsf{M}_{1, \omega_1}$.   Indeed, since $R$ is complemented in $X$, it is sufficient to show that $R$ itself lies in $\complement \textsf{M}_{1, \omega_1}$. But a normalized, asymptotic $c_0$ basis for a reflexive Banach space is $1$-weakly null and the coordinate functionals to this basis are weakly null.   Therefore such a space cannot lie in $\textsf{M}_{1, \omega_1}$. 

Similarly, if $X$ is a Banach space with a complemented, reflexive subspace $R$ having a seminormalized basis which is asymptotic $\ell_1$ in $R$, then $X$ lies in $\complement\textsf{M}_{\omega_1,1}$.   Indeed, since $R$ is complemented in $X$, it is sufficient to show that $R$ itself lies in $\complement \textsf{M}_{ \omega_1,1}$.  Arguing as in the previous paragraph, the basis of $R$ is weakly null and the coordinate functionals to the basis are asymptotic $c_0$ in $R^*$, and therefore $1$-weakly null.    Thus $R^*\in \complement \textsf{M}_{\omega_1, 1}$.

Finally, let us note that if $X$ is a Banach space with a reflexive subspace $R$ having a basis which is asymptotic $\ell_1^\alpha$ in $R$, and if $\zeta<\alpha$, then $X\in \complement \textsf{G}_{\omega_1, \zeta}$. Indeed, $R$, and therefore $X$, admits a weakly null, normalized sequence. Since $R$ is asymptotic $\ell_1^\alpha$, a subsequence of this sequence must be an $\ell_1^\alpha$-spreading model, so this sequence cannot be $\zeta$-weakly null.

\label{econ}
\end{rem}

We are now ready to prove the lower estimates on complexity in the case that at least one of the ordinals is uncountable. 

\begin{theorem}\begin{enumerate}[(i)]\item If $0\leqslant \zeta<\omega_1$, then $\mathcal{G}_{\omega_1, \zeta}$ and $\textbf{\emph{G}}_{\omega_1, \zeta}$ are $\Pi_2^1$-hard. \item If $0<\zeta\leqslant \omega_1$, then $\mathcal{M}_{\omega_1, \zeta}$, $\textbf{\emph{M}}_{\omega_1, \zeta}$, $\mathcal{M}_{\zeta, \omega_1}$, and $\textbf{\emph{M}}_{\zeta, \omega_1}$ are $\Pi_2^1$-hard. \end{enumerate}

\end{theorem}

\begin{proof} We can  deduce that the classes of operators are $\Pi_2^1$-hard if we know that the classes of spaces are $\Pi_2^1$-hard. Therefore it suffices to produce for each class $\textbf{G}_{\omega_1, \zeta}$, $\textbf{M}_{\zeta, \omega_1}$, $\textbf{M}_{\omega_1, \zeta}$ a Borel reduction of the class to $C$.

$(i)$ Let us fix $\zeta<\alpha<\omega_1$ and let $\mathfrak{S}_\alpha^*$ be the map from Corollary \ref{rob}. Then if $T\in C$, $\mathfrak{S}^*_\alpha(T)$ has the Schur property and therefore lies in $ \textbf{G}_{\omega_1, \zeta}$.    If $T\notin C$, then $\mathfrak{S}^*_\alpha(T)$ has a reflexive subspace $R$ with a basis which is asymptotic $\ell_1^\alpha$ in $R$.  We deduce that if $T\notin C$, $\mathcal{S}^*_\alpha(T)\in \textbf{SB}\setminus \textbf{G}_{\omega_1, \zeta}$ by Remark \ref{econ}. 

$(ii)$ Let us fix any $0<\alpha<\omega_1$. If $T\in C$, $\mathfrak{S}_\alpha(T)^*, \mathfrak{S}^*_\alpha(T)$ have the Schur property, so $\mathfrak{S}_\alpha(T), \mathfrak{S}^*_\alpha(T)$ lie in $\textbf{M}_{\omega_1, \omega_1}$ if $T\in C$. Now if $T\notin C$, $\mathfrak{S}_\alpha(T)$ has a complemented, reflexive subspace $R$ with asymptotic $c_0^1$-basis, so $\mathfrak{S}_\alpha(T)\in \textbf{SB}\setminus \textbf{M}_{1, \omega_1}\subset \textbf{SB}\setminus \textbf{M}_{\zeta, \omega_1}$ by Remark \ref{econ}.   A similar appeal to Remark \ref{econ} yields that if $T\notin C$, $\mathfrak{S}^*_\alpha(T)$ has a complemented, reflexive subspace $R$ with asymptotic $\ell_1^1$-basis, so $\mathfrak{S}^*_\alpha(T)\in \textbf{SB}\setminus  \textbf{M}_{\omega_1, 1}$.

\end{proof}

We now complete the lower estimate on the complexity in the case of two countable ordinals.  This follows from a standard tree space construction.

\begin{proposition} \begin{enumerate}[(i)]\item For $0\leqslant \zeta<\xi\leqslant \omega_1$, $\mathcal{G}_{\xi, \zeta}$ and $\textbf{\emph{G}}_{\xi, \zeta}$ are $\Pi_1^1$-hard.  \item For $1\leqslant \zeta, \xi\leqslant \omega_1$, $\mathcal{M}_{\xi, \zeta}$ and $\textbf{\emph{M}}_{\xi, \zeta}$ are $\Pi_1^1$-hard. \end{enumerate}

\end{proposition}

\begin{proof} We prove the result for spaces.  We prove $(i)$ and $(ii)$ simultaneously. First fix any space $Y$ with a normalized, bimonotone basis $(y_i)_{i=1}^\infty$ such that $Y$ lies in $\complement\textsf{G}_{\xi, \zeta}$ (resp. $\complement\textsf{M}_{\xi, \zeta}$).    Note that $X_\zeta\in \complement\textsf{G}_{\xi, \zeta}$, as the canonical basis is $\zeta+1$-weakly null and not $\zeta$-weakly null, and $\ell_2\in \complement \textsf{M}_{1,1}\subset \complement \textsf{M}_{\xi, \zeta}$. So such a $Y$ exists.

Now let $\mathcal{T}$ denote the finite, non-empty sequences of natural numbers and for such a sequence, let $|t|$ denote the length of $t$.  Let us define the relation $\preceq$ on $\mathcal{T}$ by $s\preceq t$ if $s$ is an initial segment of $t$.  We say a subset $\mathfrak{s}\subset \mathcal{T}$ is a \emph{segment} if it is of the form $\mathfrak{s}=\{v: u\preceq v\preceq w\}$ for some $u,w\in \mathcal{T}$.     Let us say two segments $\mathfrak{s}_0, \mathfrak{s}_1$ are \emph{incomparable} if for $j\in \{0,1\}$, no member of $\mathfrak{s}_j$ is an initial segment of any member of $\mathfrak{s}_{1-j}$.        Now let $Z$ denote the completion of $c_{00}(\mathcal{T})$ with respect to the norm $$\|\sum_{t\in \mathcal{T}}a_t e_t\|=\sup\Bigl\{\sum_{i=1}^n \|\sum_{t\in \mathfrak{s}_i} a_t y_{|t|}\|_Y: n\in\nn, \mathfrak{s}_1, \ldots, \mathfrak{s}_n\subset \mathcal{T}\text{\ pairwise incomparable segments}\Bigr\}.$$    For a tree $T$ on $\nn$, let $Z(T)$ denote the closed span in $Z$ of $\{e_t: t\in T\setminus\{\varnothing\}\}$.    Then by an easy induction on the rank of $T$, if $T$ is well-founded, $Z(T)$ has the Schur property and therefore lies in $\textsf{G}_{\xi, \zeta}$ (resp. $\textsf{M}_{\xi, \zeta}$).     If $T$ is ill-founded, $Z(T)$ contains a complemented copy of $Y$, and therefore lies in $\complement \textsf{G}_{\xi, \zeta}$ (resp. $\complement \textsf{M}_{\xi, \zeta}$).    Furthermore, by standard techniques (see, for example, \cite{Bossard}), there exists a map $J:\text{Subs}(Z)\to \textbf{SB}$ mapping the space of closed subspaces of $Z$ into $\textbf{SB}$ such that $T\mapsto J(Z(T))$ is Borel. This is a reduction of $\textbf{WF}$ to $\textbf{G}_{\xi, \zeta}$ (resp. $\textbf{M}_{\xi, \zeta}$).

\end{proof}


\begin{thebibliography}{HD}

\normalsize
\baselineskip=17pt


\bibitem{AG} S. A. Argyros,  I. Gasparis,  \emph{Unconditional structures of weakly null sequences},  Trans. Amer. Math. Soc.
\textbf{353}(5) (2001), 2019-2058.



\bibitem{AMT} S. A. Argyros, S. Mercourakis, and A. Tsarpalias, \emph{ Convex unconditionality and summability of weakly null
sequences},  Israel J. Math., \textbf{107} (1998), 157-193. 

\bibitem{BC2} K. Beanland, R.M. Causey, \emph{Quantitative factorization of weakly compact, Rosenthal, and $\xi$-Banach-Saks operators}, to appear in Mathematica Scandinavica. 

\bibitem{BC} K. Beanland, R.M. Causey, \emph{Genericity and universality for operators ideals}, submitted. 

\bibitem{BF} K. Beanland,  D. Freeman, \emph{ Ordinal ranks on weakly compact and Rosenthal operators},  Extracta Math. \textbf{26} (2011), no. 2, 173-194. 


\bibitem{Bossard} B. Bossard, \emph{Th\'{e}orie descriptive des ensembles en g\'{e}om\'{e}trie des espaces de Banach},  PhD thesis, 1994.

\bibitem{Bourgain}  J. Bourgain, \emph{The metrical interpretation of super-reflexivity in Banach spaces}, Israel J.
Math. \textbf{56} (1986) 221-230.



\bibitem{C2} R.M. Causey, \emph{Concerning the Szlenk index},  Studia Math., \textbf{236} (2017), 201-244.    

\bibitem{Causey1} R.M. Causey, \emph{The $\xi,\zeta$-Dunford Pettis property}, preprint. 

\bibitem{CD} R.M. Causey, S.J. Dilworth, \emph{Metrical characterizations of super weakly compact operators},  Studia Math., \textbf{239} (2017), no. $2$, 175-188. 

\bibitem{CN} R.M. Causey, K. Navoyan, \emph{$\xi$-completely continuous operators and $\xi$-Schur Banach spaces}, J. Funct. Anal., DOI 10.1016/j.jfa.2018.09.002. 

\bibitem{Dodos} P. Dodos, \emph{Banach spaces and descriptive set theory: selected topics}, volume 1993 of Lecture Notes in Mathematics.
Springer-Verlag, Berlin, 2010.

\bibitem{Gasparis} I. Gasparis, \emph{ A dichotomy theorem for subsets of the power subsets of the power set of the natural
numbers}, Proc. Am. Math. Soc. \textbf{129}, (2001), 759-764.

\bibitem{GJ} M. Girardi,  W. B. Johnson, \emph{Universal non-completely-continuous operators}, Israel J. Math., \textbf{99} (1997), 207-219. 

\bibitem{Johnson}  W. B. Johnson, \emph{ A universal non-compact operator},  Colloq. Math., \textbf{23} (1971), 267-268. 

\bibitem{LP} J. Lindenstrauss,  A. Pe \l czy\'{n}ski, \emph{ Absolutely summing operators in $L_p$-spaces and their applications},  Studia
Math., \textbf{29} (1968), 275-326.



\bibitem{Ke} A. S. Kechris, \emph{ Classical descriptive set theory}, volume 156 of Graduate Texts in Mathematics. Springer-Verlag,
New York, 1995.

\bibitem{KRN}  K. Kuratowski, C. Ryll-Nardzewski, \emph{A general theorem on selectors}, Bull Acad. Pol.
Sci. S\'{e}r. Sci., Math. Astr. et Phys. \textbf{13} (1965), 397-403.


\bibitem{K} O. Kurka, \emph{Tsirelson-like spaces and complexity of classes of Banach spaces}, RACSAM (2017). https://doi.org/10.1007/s13398-017-0412-9

\bibitem{Oikhberg} T. Oikhberg, \emph{A note on Universal operators}, Ordered Structures and Applications. Birkh\"{a}user, Basel, 2016. 

\end{thebibliography}
\end{document}